%% file: nearfield.tex
\documentclass[review]{elsarticle}
\usepackage{color}
\usepackage{lipsum} 
\usepackage{xcolor}
\usepackage{amsmath,amsthm,amssymb}
\usepackage{amsfonts}
\usepackage{amscd}

\usepackage{graphicx}
\usepackage{enumitem}
\usepackage{mathtools}

\DeclareMathOperator*{\Proj}{Proj}
\DeclareMathOperator*{\Int}{Int}
\usepackage{enumitem}

\usepackage{amsmath}
\usepackage{amsfonts}
\usepackage{amssymb}
\usepackage{amsthm}
\usepackage{amscd}
\usepackage[hidelinks]{hyperref}
\usepackage[ruled,linesnumbered]{algorithm2e}
\usepackage[numbers]{natbib}
\usepackage{mathrsfs}
\usepackage{color}
\usepackage{lipsum} 
\usepackage{xcolor}

\newcounter{cases}
\newcounter{subcases}[cases]
\newenvironment{mycase}
{
    \setcounter{cases}{0}
    \setcounter{subcases}{0}
    \newcommand{\case}
    {
        \par\indent\stepcounter{cases}\textbf{Case \thecases.}
    }
    
}
{
    \par
}
\renewcommand*\thecases{\arabic{cases}}

\journal{Constructive Approximation}

\begin{document}
\newcommand{\R}{\mathbb{R}}
\newcommand{\origin}{\mathcal{O}}
\newcommand{\sphere}{\mathbb{S}}

\newtheorem{Def}{Definition}[section]
\newtheorem{claim}{Claim}[section]
\newtheorem{lemma}{Lemma}[section]
\newtheorem{prop}{Proposition}[section]
\newtheorem{conj}{Conjecture}[section]
\newtheorem{theorem}{Theorem}[section]
\newtheorem{coro}{Corollary}[section]

\begin{frontmatter}

\title{Weak solutions to the near-field reflector problem with spatial restrictions approached with generalized reflectors constructed from ellipsoids}

\author{Dylanger S. Pittman}
\address{400 Dowman Drive, Atlanta}

\ead{dpittm2@emory.edu}

\begin{abstract}
We motivate then formulate a novel variant of the near-field reflector problem and call it the near-field reflector problem with spatial restrictions. Let $\origin$ be an anisotropic point source of light and assume that we are given a bounded open set $U$. Suppose that the light emitted from the source at $\origin$ in directions defined by the aperture $D\subseteq \sphere^2$, of radiance $g(m)$ for $m\in D$, is reflected off $R\subset \overline{U}$, creating the irradiance $f(x)$ for $x\in T$. The inverse problem consists of constructing the reflector $R\subseteq \overline{U}$ from the given position of the source $\origin$, the input aperture $D$, radiance $g$, `target' set $T$, and irradiance $f$. We focus entirely on the case where the target set $T$ is finite. 
\end{abstract}


\end{frontmatter}

\section{Introduction}

Let $\origin$ be the origin of $\R^3$, and let $\sphere^2$ be the unit sphere centered at $\origin$. We treat points on $\sphere^2$ as unit vectors with initial points at $\origin.$ Let an {\it aperture} be a subset of $\sphere^2$; in our work, the aperture will be an open set. Physically, it makes sense to consider $\origin$ as the location of an anisotropic point source of light such that rays of light are emitted in a set of directions defined by an aperture $D\subseteq \sphere^2$.

\begin{Def}\label{reflectorDef}
    Assume that we are given an aperture that is a connected open set $D\subseteq \sphere^2$, and a function $\rho:D\to (0,\infty)$ that is continuous and almost everywhere differentiable. Then a {\bf reflector} is the set $R=\{m\rho(m)|m\in D\}\subset \R^3$. 
    
    If $\rho$ is a smooth function, we can call $R$ a {\bf smooth reflector}.
\end{Def}

 Given an aperture, $D$, that is a connected open set, assume that we have a continuous, almost everywhere differentiable, positive function $\rho:D\to (0,\infty)$, and a corresponding reflector $R=\{m\rho(m)|m\in D\}.$ Suppose that a ray originating from $\origin$ in the direction $m\in D$ is incident on the reflector $R$ at the point $m\rho(m)$. If $\rho$ is differentiable at $m$, there is a unit vector, $n(m)$, normal to the reflector $R$ at $m\rho(m)$. Therefore, by the reflection law of geometric optics, a ray from $\origin$ of direction $m$ reflects off the point $m\rho(m)$ in the direction 
\begin{equation}\label{geo_reflect}
    y(m)=m-2\langle m,n(m)\rangle n(m)
\end{equation}
where $\langle m,n(m)\rangle$ is the standard Euclidean inner product in $\R^3$ and $n(m)$ is oriented such that $\langle m,n(m)\rangle>0$ \cite{born_wolf_2019}. 

The reflector $R$ is designed such that the ray described by the point $m\rho(m)\in R$ and the direction $y(m)$ corresponds to some element in a prespecified {\it target set} $T$. What one means by a `target set' changes depending on the context, and the correspondence between $y(m)$; also, an element of the target set can also vary depending on one's needs. Hence a target set can represent many things. For example, if the target set $T$ is a subset of $\sphere^2$, then a possible correspondence can be $\frac{y(m)}{|y(m)|}\in T$; see \cite{para}. Physically, in this case, $T$ can be considered as a set of directions for rays of light. If $T$ is a subset $\R^3\setminus \{\origin\}$, then for an example of another possible correspondence, we can say that for every $m\in D$, there exists an $a(m)>0$ such that $a(m)y(m)+m\rho(m)\in T;$ see \cite{Schruben:72} and \cite{OK1}. Physically, in this case, $T$ can be considered as a region that one wants to illuminate.

Assume that $g$ is an integrable and nonnegative function over an aperture $D$, and $f$ is an integrable and nonnegative function over a target set $T$. Physically speaking, we say $g(m)$ for $m\in D$ is the radiance of the source at $\origin$ in the directions $m\in D$, or that $g$ is a radiance distribution over $D$. We also say $f(x)$ for $x\in T$ is the irradiance of the target set at $x\in T$, or that $f$ is an irradiance distribution over $T$.

A {\it reflector system} comprises of an aperture $D$, $\origin$, a reflector $R$, an integrable and nonnegative function $g$ over $D$, and a target set $T$ with an integrable and nonnegative function $f$ over $T$. From a physical perspective: light emitted from the source at $\origin$ in directions defined by the aperture $D$, of radiance $g(m)$ for $m\in D$, is reflected off $R$, creating the irradiance $f(x)$ for $x\in T$. An example that can serve as an illustration is shown in Figure \ref{figure_plane-reflector}.

\begin{figure}[htbp]
\centering
\label{figure_plane-reflector}
\def\svgwidth{0.9\textwidth}
\input{figure_plane-reflector.pdf_tex}
\caption[Reflector System with a Planar Reflector Illustration]{Here is the most basic example of a reflector system with a smooth reflector. Here $R$ is a plane. Every point on $R$ has a normal. Light originates from the point $\origin$ with directions represented by points on the unit sphere $\sphere^2$ and travels according to some target set that is neither shown nor specified.}
\end{figure}
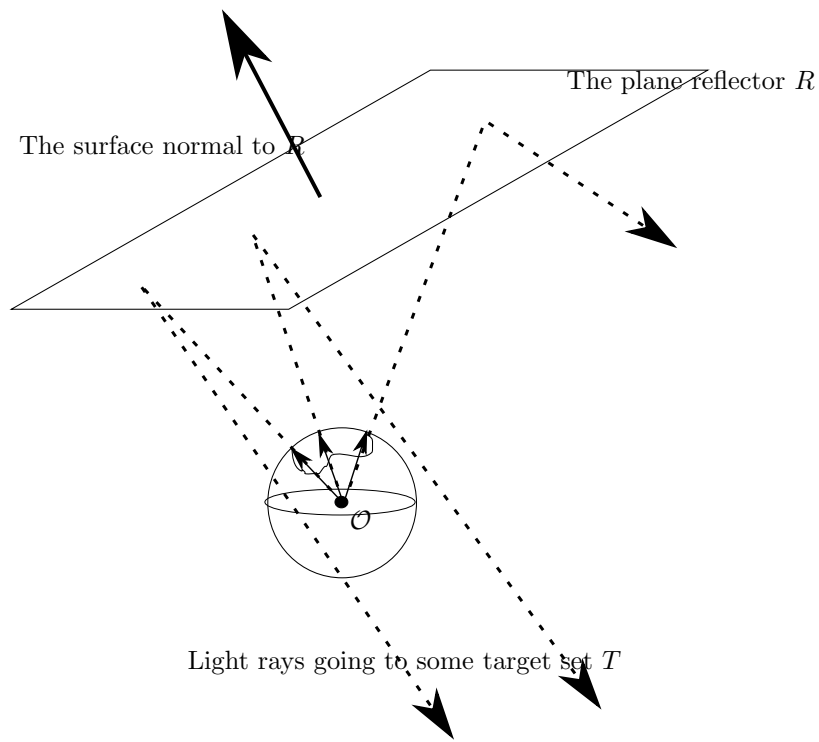

A {\it reflector problem} is, in short, an inverse problem that seeks to complete a reflector system by creating a reflector that fits the other information given. Specifically, suppose we are given $\origin$, an aperture $D$, an integrable and nonnegative function $g$ over $D$, and a target set $T$ with an integrable and nonnegative function $f$ over $T$. The aim of a reflector problem is to find a continuous, almost everywhere differentiable, positive $\rho$ over $D$ such that the reflector $R=\{m\rho(m)|m\in D\}$ produces the specified in advance irradiance distribution $f$ on $T$.

Reflector problems have been well studied due to their utility in physics and engineering. Such problems have found numerous applications in the construction of reflector antennas (see \cite{XuWang}, \cite{kinber}), mirror design \cite{BurkShealy}, heat transfer \cite{HorMc}, and beam shaping \cite{10.1117/12.897108}. We only consider in the high-frequency approximation of light, where the laws of geometric optics apply. We now proceed with a general description and motivation for the near-field reflector problem.

\section{The Near-Field Reflector Problem}\label{nearfield-formulation}

We discuss a reflector problem that we call the `near-field reflector problem.' In short, the near-field reflector problem aims to design a reflector that redistributes the light from the origin onto a set a finite distance away from the origin. 

In this part, when we say {\it surface}, we mean it in the differential geometric sense; see Definition 12.4 in \cite{diffgeo}. Suppose that we are given a reflector system consisting of 
\begin{enumerate}
    \item $\origin$,
    \item an aperture $D\subset \sphere^2$,
    \item a nonnegative $g\in L^1(D)$,
    \item a bounded Borel set $T\subset \R^3\setminus\{\origin\}$ (typically either a subset of a surface or a finite set),
    \item a nonnegative and integrable function $f:T\to [0,\infty)$,
    \item and a smooth function $\rho:D\to(0,\infty)$ with a smooth reflector $R=\{m\rho(m)|m\in D\}.$ 
\end{enumerate}
From a physical perspective, this setup can be described as follows. The light is emitted from the source at $\origin$ in directions defined by the aperture $D$. Each ray of direction $m\in D$ has radiance $g(m)$ and is reflected off $R$ at the point $m\rho(m)$ in the direction $y(m)$ as described by (\ref{geo_reflect}). For every $m\in D$, there exists an $a(m)>0$ such that $a(m)y(m)+m\rho(m)\in T$ creating the irradiance $f(x)$ for $x\in T$. A basic illustration of this situation is depicted in Figure \ref{figure_near-field}. With this setup in mind, we proceed with a formulation of the near-field reflector problem

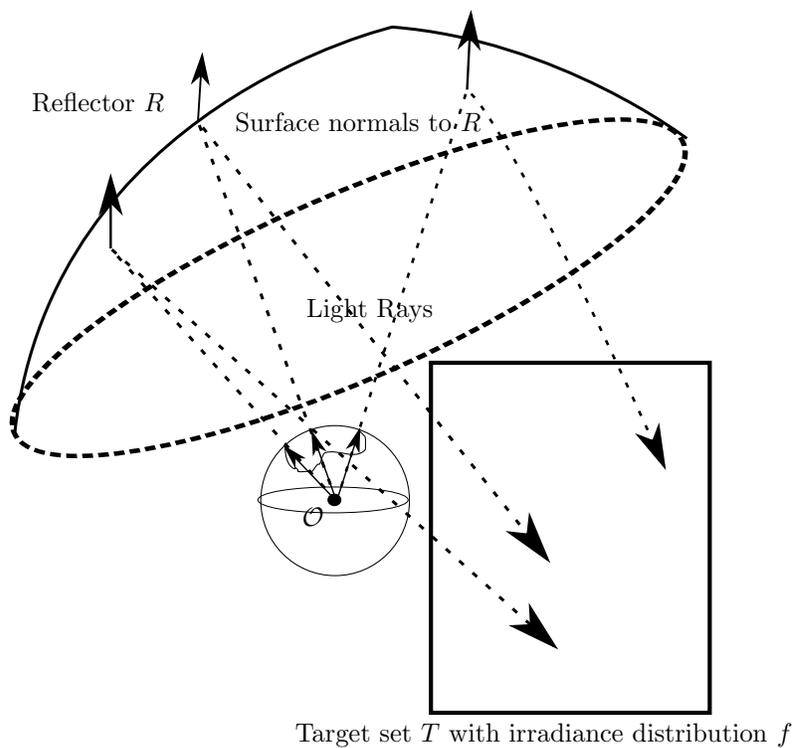
\begin{figure}[htbp]
\centering
\label{figure_near-field}
\def\svgwidth{0.9\textwidth}
\input{figure_near-field-reflector.pdf_tex}
\caption[Near-Field Reflector Illustration]{Here is an illustration of the near-field reflector problem in $\R^3$. The radiation intensity at the origin $\origin$ is given by a nonnegative function $g\in L^1(D)$. We want to find a reflector $R$ such that the reflected rays produce the prescribed irradiance distribution $f$ on $T.$}
\end{figure}

Let $u=(u^1,u^2)$ be smooth local coordinates on $\sphere^2$ such that $D$ lies in one coordinate patch. The position vector of a point $m \in D$ is $m = m(u)$. We choose the coordinates $u^1, u^2$ so that $\langle m, m_1 \times m_2\rangle =1$ in $D$; here, $\langle,\rangle$ denotes the scalar product in $\R^3$ and $m_i = \frac{\partial m}{\partial u^i}, i = 1,2.$ Observe that this implies that  $\langle m,m_i \rangle=0$, $i=1,2.$ The first fundamental form of $\sphere^2$ is given by $e = e_{ij} du^i du^j$ where $e_{ij} = \langle m_i , m_j\rangle.$ 

Set $r(m)=m\rho(m)$, then $r(m)$ defines a smooth surface $R=\{r(m)|m\in D\}$. Let $g=g_{ij}du^idu^j$ be the first fundamental form of $R$ where $g_{ij}=\langle r_i,r_j\rangle=\rho_i\rho_j+\rho^2e_{ij}$, $r_i=\frac{\partial r}{\partial u_i}$, and $\rho_i=\frac{\partial \rho}{\partial u_i}$.

Let $n(m)$ is the normal vector field on $R$ such that $\langle n(m),m\rangle>0$ everywhere on $R.$ Then
\begin{equation}\label{norm-near-field}
    n(m)=(\rho^2+|\tilde{\nabla}\rho|^2)^{-1/2}(r-\tilde{\nabla}\rho)
\end{equation}
where $|\tilde{\nabla}p|^2=\rho_i\rho_je^{ij}.$
This combined with equation (\ref{geo_reflect}) determines the direction a ray will go after reflecting off $R$ \cite{OLIKER199391}. 

We can now track the path of each ray described by the direction $m\in D$ to a point $x(m)\in T$. A ray, originating at $\origin$ in direction $m$, hits the surface $R$ at a point $r(m)$. Then, said ray reflects off $R$ at $r(m)$ in the direction $y(m)$ as defined by (\ref{geo_reflect}) and reaches $T$ at some point $x(m)$. Thus, from a physical perspective, an irradiance $f(x(m))$ is created by the rays reflected at $x(m)$. This defines a mapping $m\to x$ that we call a {\it reflector map}; for convenience, we denote $x(m)$ as the image of $m$ under the reflector map. The reflector map $x:D\to T$ combined with equations (\ref{geo_reflect}) and (\ref{norm-near-field}) describes the ray tracing from $D$ to $T$.

If the reflector map is a diffeomorphism from $D$ to $T$ where $T$ is a subset of a smooth surface, then one can introduce the first fundamental form of $T$ as $w =w_{ij} du^i du^j ,$ where $w_{ij} = \langle x_i ,x_j\rangle, x_i = \frac{\partial x}{\partial u^i}.$

According to the differential form of the energy conservation law \cite{born_wolf_2019}, 
\begin{equation}\label{diff-near-field-nrg}
    f(x(m)) |J(x(m))| = g(m)
\end{equation}
where $J$ is the Jacobian determinant of the map $x$. Note that
\begin{equation}
    J(x(m))=\pm\frac{d\nu(x(m))}{d\sigma(m)}=\pm\frac{\sqrt{\det(w_{ij})}}{\sqrt{\det(e_{ij})}}
\end{equation}
where $d\sigma$ is the surface area element on $\sphere^2$, and $d\nu$ is the surface area element on $T$. We assign a $\pm$ sign to the Jacobian according to whether $x$ preserves the orientation or reverses it. Therefore, by integration of (\ref{diff-near-field-nrg}), for all Borel sets $\omega \subseteq T$,
\begin{equation}
    \int_{x^{-1}[\omega]} g d\sigma= \int_{\omega} f d\nu
\end{equation}
where $x^{-1}[\omega]=\{m\in D|x(m)\in\omega\}$ and $\int_Dg d\sigma=\int_T fd\nu$. 

With this motivation, we can now state the near-field reflector problem. Assume that we are given $\origin$, an aperture $D\subset \sphere^2$ with a nonnegative function $g\in L^1(D)$, and a bounded Borel set $T\subset \R^3\setminus\{\origin\}$ with a nonnegative, integrable function $f:T\to [0,\infty)$. The goal is to find a smooth function $\rho$ over $D$ such that:
\begin{enumerate}
    \item The ray originating from $\origin$ in the direction $m\in D$ reflects off the reflector $R=\{m\rho(m)|m\in D\}$ in accordance with equation (\ref{geo_reflect}) and reaches the target set $T$.
    
    \item $g(m)$ on $D$ is transformed by the reflector map into $f$ on $T$; i.e. for all Borel subsets $\omega \subseteq T$,
\begin{equation}\label{nrg-conserv1}
    \int_{x^{-1}[\omega]} g d\sigma= \int_{\omega} f d\nu
\end{equation}
where $x: D\to T$ the reflector map corresponding to the reflector $R=\{m\rho(m)|m\in D\}$, $x^{-1}[\omega]=\{m\in D|x(m)\in\omega\}$, $d\sigma$ is the surface area element on $\sphere^2$, and $d\nu$ is the area element on $T$ ($\nu$ is typically some discrete or Lebesgue measure).

\item The law of total energy conservation is obeyed: $\int_Dg d\sigma=\int_T fd\nu$.
\end{enumerate}

The case where the reflector map is a diffeomorphism from $D$ to $T$ can be alternatively formulated as a PDE of Monge-Amp\`ere type; specifically equation (4) from \cite{Oliker_1989}.

There has been a lot of work done on the near-field reflector problem. In 1972, Schruben \cite{Schruben:72} found that if the target set was a subset of a plane in $\R^3$, one can then derive an implicit integro-differential equation describing the reflector; the existence of a solution was not proved. Then in \cite{Schruben:74}, Schruben considered the case where the target set was a small rotationally symmetric patch on the plane. In this case, when the radiance and the irradiance distributions are rotationally symmetric, the equation derived in \cite{Schruben:72} can be solved as an ODE. In 1989, Oliker \cite{Oliker_1989} found a formulation of the near-field reflector problem in the form of a strongly non-linear PDE of Monge-Amp\`ere type. The exploration of the said equation is difficult and in \cite{Oliker_1989} was solved only for the rotationally symmetric case. 

In 1998, Kochengin and Oliker \cite{OK1} introduced an alternative formulation to the near-field reflector problem, which was a geometric approach involving the analysis of the boundaries of convex sets generated by families of supporting ellipsoids. This approach can also be considered a weak solution to the PDE introduced in \cite{Oliker_1989}. The strategy was to assume that the target set was a finite set on a plane and constructively prove the existence of solutions for that case. Since the reflectors that were constructed were convex, one can use the Blaschke selection theorem (for more details, see \cite{schneider_2013}) to prove the existence of a solution with a continuous target set on the plane. This method was largely motivated by previous work done by Caffarelli and Oliker \cite{para} which involved the analysis of the boundaries of convex sets generated by families of supporting paraboloids to solve a related problem. 

In \cite{OK2} a provably convergent numerical algorithm was introduced that explicitly finds the ellipsoids required to construct the reflectors described in \cite{OK1}. It was shown that this construction leads to infinitely many solutions; however, the algorithm has the benefit of converging to a unique solution if we fix an initial point on the reflector. This algorithm and its variations have been explored extensively in various scenarios. For example, Fournier, Cassarly, and Rolland in \cite{fournier} adapted the algorithm in \cite{OK2}, to situations where the light source is not a single point; specifically, a flat rotationally symmetric emitter. In \cite{Fournier:10} a method was proposed for smoothing out a reflector with a discrete irradiance distribution to a reflector with a continuous irradiance distribution. Optimal transport methods have also been studied \cite{Graf2012AnOM}. 

\subsection{The Near-Field Reflector Problem with Spatial Restrictions}

In this paper, we study a novel variant of the near-field reflector problem where we have extreme limitations on where we can place and construct the reflectors. Specifically, we are given an open set $U\subset \R^3\setminus\{\origin\}$, and our reflector $R$ must now be a subset of $\overline{U}$. 

\begin{Def}\label{proj-def}
Given an $x\in\R^3\setminus\{\origin\}$ and a subset $S\subseteq \R^3\setminus\{\origin\}$, then we define $\Proj(x)=\frac{x}{|x|}$ as {\bf the projection of $x$ onto $\sphere^2$} and $\Proj[S]=\{\Proj(x)\in \sphere^2|x\in S\}$ as {\bf the projection of $S$ onto $\sphere^2$}.
\end{Def}

Assume that we are given a positive, continuous, almost everywhere differentiable function $\rho$ over $\Proj[U]$. We have a reflector $R=\{m\rho(m)|m\in \Proj[U]\}$ which determines our reflector map $x: \Proj[U]\to T$ which is determined by tracking the path of each ray described by the direction $m\in \Proj[U]$ to a point $x(m)\in T$. A ray, originating at $\origin$ in direction $m$, hits the reflector $R$ at a point $m\rho(m)$. Then, assuming $\rho$ is differentiable at $m$, said ray reflects off $R$ at $m\rho(m)$ in the direction $y(m)$ as defined by (\ref{geo_reflect}) and reaches $T$ at some point $x(m)$. Thus, from a physical perspective, an irradiance $f(x(m))$ is created by the rays reflected at $x(m)$. This defines a mapping $m \to x(m)$ that we call the {\it reflector map}; for convenience, we denote $x(m)$ as the image of $m$ under the reflector map.

We can now formulate the near-field reflector problem with spatial restrictions. Assume that we are given an open set $U\subset \R^3\setminus\{\origin\}$, $\origin$, an aperture $\Proj[U]\subset \sphere^2$, a nonnegative $g\in L^1(\Proj[U])$, and a bounded Borel set $T\subset \R^3\setminus\{\origin\}$ with an integrable function $f:T\to [0,\infty)$. 

The goal is to find a positive, continuous, almost everywhere differentiable function $\rho$ over $\Proj[U]$ such that:
\begin{enumerate}
\item $R=\{m\rho(m)|m\in\Proj[U]\}\subset \overline{U}$.
    \item The ray originating from $\origin$ in the direction $m\in \Proj[U]$ reflects off of $R$ in accordance with equation (\ref{geo_reflect}) and reaches the target set $T$.
    
    \item $g(m)$ on $\Proj[U]$ is transformed by the reflector map into $f$ on $T$, i.e. for all Borel subsets $\omega \subseteq T$,
\begin{equation}\label{nrg-conserv1a}
    \int_{x^{-1}[\omega]} g d\sigma= \int_{\omega} f d\nu
\end{equation}
where $x:\Proj[U]\to T$ is the reflector map, $x^{-1}[\omega]=\{m\in \Proj[U]|x(m)\in\omega\}$, $d\sigma$ is the surface area element on $\sphere^2$, and $d\nu$ is the area element on $T$ ($\nu$ is typically some discrete or Lebesgue measure). 

\item The law of total energy conservation is obeyed: $\int_{\Proj[U]}g d\sigma=\int_T fd\nu$.
\end{enumerate}

This variation of the near-field reflector problem has clear applications to engineering; as often one has to grapple with restrictions of space in real-world designs. For example, in the construction of automotive headlights, there are strict restrictions, guided purely by aesthetics, as to where a reflector can be placed and how a reflector must be shaped \cite{FlorianDiss}. However, to the author's knowledge, no mathematical research has been done in this direction. We focus exclusively on the case where the target set is finite. 

\section{Ellipsoids of Revolution}

We do all our work in $\R^3$. We denote $\sphere^2$ to be the unit sphere with the center at $\origin$ and $k_x=x/|x|$ for all $x\in \R^3\setminus\{\origin\}.$ We borrow much of this geometric setup from \cite{OK1} and \cite{OK2}. Ellipsoids of revolution are of paramount importance when solving the near-field reflector problem due to their unique optical properties.

Let $x\in \R^3\setminus\{\origin\}$ and $d\in(0,\infty)$. We denote by $E_d(x)$ an ellipsoid of revolution about the axis $\origin x$ and with foci at points $\origin$ and $x$. The polar radius relative to $\origin$ can be represented as:
\begin{equation}
    \psi_{x,d}(m)=\frac{d}{1-\epsilon \langle m,k_x \rangle}, \; m \in \sphere^2
\end{equation}
where $\epsilon$ is the eccentricity and 
\begin{equation}
    \epsilon = \sqrt{1+\frac{d^2}{x^2}}-\frac{d}{|x|}.
\end{equation}
So in other words
\begin{equation}
    E_d(x)=\{m\psi_{x,d}(m)|m\in \sphere^2\}.
\end{equation}
From this point on, whenever we use the term \textit{ellipsoid} we specifically refer to an ellipsoid of this kind with one of the foci always at $\origin$. Note that each $E_d(x)$ is uniquely defined by the $x\in \R^3\setminus\{\origin\}$ and the $d\in (0,\infty).$ In this paper, we define $\Psi_{x,d}(m)=m\psi_{x,d}(m)$. 

Note that for all possible values of $d$, we have that $\epsilon\in(0,1)$. Also for a fixed $x$, as $d\to 0$ the ellipsoid will degenerate into a line segment, i.e. $E_d(x)\to\{tx+(1-t)\origin|t\in[0,1]\}.$ Such an ellipsoid is called {\it degenerate}. Observe that as $d\to \infty$, $|\psi_{x,d}(m)|\to\infty$ for all $m\in\sphere^2.$ 

An important property of ellipsoids can be described by the following proposition. 

\begin{prop}
Let $c,d>0$. Then the ellipsoids $E_{cd}(x)$ and $E_{d}(x)$ have the same foci: $\origin$ and $x.$
\end{prop}

From a physical perspective, the aforementioned property is important because a reflector that is shaped like an ellipsoid $E_d(x)$ will illuminate the focus $x$ with the light emitted from $\origin$ such that the total energy emitted from $\origin$ is equal to the total energy reflected onto $x$. This property is still true no matter how large or small the ellipsoid is; all that matters is the location of the foci.

\section{Generalized Reflectors}
Before we proceed, we reiterate that the near-field reflector problem can be expressed analytically as a PDE of Monge Amp\'ere Type. Specifically, the equation (4) from \cite{Oliker_1989}. Therefore we will consider the following formulation of the near-field reflector problem with spatial restrictions as a weak formulation and its solutions, weak solutions. The following formulation only concerns the case where the target set is finite.

\subsection{Weak Solutions Using Generalized Reflectors}

\begin{Def}\label{genralizedref}
    Assume that we are given an aperture $D\subseteq \sphere^2$ that is an open set, and a function $\rho:D \to (0,\infty)$ that is not necessarily continuous and almost everywhere differentiable. Then a {\bf generalized reflector} is the set $R=\{m\rho(m)|m\in D\}\subset \R^3$.
\end{Def}

The upper half-space of $\R^3$ be represented as $\R^{3+}=\{(x,y,z)\in \R^3|z>0\}$, and the lower half-space of $\R^3$ be represented as $\R^{3-}=\{(x,y,z)\in \R^3|z<0\}$. Let $\sigma$ denote the standard measure on $\mathbb{S}^2.$ Consider an open set $U\subseteq \R^{3+}$, a corresponding aperture $\Proj[U]$, and a finite target set $T\subset \R^{3-}$. 

Let $\mathscr{B}$ be a countable family of open subsets of $\sphere^2$ such that $\sigma(\overline{\Proj[U]}\setminus\bigcup_{B\in \mathscr{B}}\overline{B})=0$, $\Proj[U]\subseteq\bigcup_{B\in \mathscr{B}}\overline{B}$, and $\sigma(\overline{B}\cap \overline{B'})=0$ for all distinct $B,B'\in \mathscr{B}$. Let the set $\mathcal{B}(U)$ be the set of all such families.

Since every ellipsoid requires foci and an eccentricity to be well defined, given a family $\mathscr{B}\in \mathcal{B}(U)$, let $\mathscr{U}_T(\mathscr{B})$ be the set of all functions $\mathscr{B}\to T$ and $\mathscr{V}(\mathscr{B})$ be the set of all functions $\mathscr{B}\to (0,\infty)$. Thus we define
\begin{equation}
    \mathcal{E}_T(U)=\left\{\left.\bigcup_{B\in \mathscr{B}}\Psi_{u(B),v(B)}[\overline{B}]\right| \mathscr{B}\in \mathcal{B}(U),u\in\mathscr{U}_T(\mathscr{B}),v\in\mathscr{V}(\mathscr{B})\right\}.
\end{equation}

Assume we are given a $Z\in \mathcal{E}_T(U)$. Let us define 
\begin{equation}
    \mathscr{B}_Z=\left\{\Int(\Proj[E_d(x)\cap Z])\subseteq \sphere^2\left|d\in(0,\infty),x\in T, \sigma(\Proj[E_d(x)\cap Z])\neq 0\right.\right\}.
\end{equation}
The geometry of the ellipsoid and the definition of $\mathscr{B}_Z$ imply that there exists unique $u\in\mathscr{U}_T(\mathscr{B})$ and $v\in\mathscr{V}(\mathscr{B})$ such that $Z=\bigcup_{B\in \mathscr{B}_Z}\Psi_{u(B),v(B)}[\overline{B}]$. Define $u_Z\in \mathscr{U}_T(\mathscr{B}_Z)$ and $v_Z\in \mathscr{V}(\mathscr{B}_Z)$ be the unique functions such that  $Z=\bigcup_{B\in \mathscr{B}_Z}\Psi_{u_Z(B),v_Z(B)}[\overline{B}]$. Given some $Z\in \mathcal{E}_T(U)$, let $y_Z^1(m)=\{B\in\mathscr{B}_Z|m\in \overline{B}\}$ for $m\in \Proj[U]$. Given a $\mathscr{B}\in \mathcal{B}(U)$, let $\mathscr{N}(\mathscr{B})$ be the set of all injective functions $s:\mathscr{B}\to \mathbb{N}$. For $Z\in \mathcal{E}_T(U)$ and $s\in\mathscr{N}(\mathscr{B}_Z)$, define
\begin{equation}
    \rho_Z^s(m)= \psi_{u_Z(s^{-1}(\min s[y_Z^1(m)])),v_Z(s^{-1}(\min s[y_Z^1(m)]))}(m) , \; m \in \Proj[U].
\end{equation}
Observe that the function $\rho_Z^s$ is positive, not necessarily continuous, and almost everywhere differentiable. Let $W(\rho_Z^s)=\{m\rho_Z^s(m)|m\in\Proj[U]\}$ and thus we describe a set of generalized reflectors
\begin{equation}\label{gen-ref-set}
    \mathcal{R}_1^U(T)=\left\{\left.W(\rho_Z^s)\right|Z\in \mathcal{E}_T(U)\text{ where }Z\subset \overline{U}, s\in\mathscr{N}(\mathscr{B}_Z) \right\}.
\end{equation}

Assume we are given a generalized reflector $R\in \mathcal{R}_1^U(T)$. Let us define 
\begin{equation}
    \mathscr{B}_R=\left\{\Int(\Proj[E_d(x)\cap R])\subseteq \sphere^2\left|d\in(0,\infty),x\in T, \sigma(\Proj[E_d(x)\cap R])\neq 0\right.\right\}.
\end{equation}
The geometry of the ellipsoid and the definition of $\mathscr{B}_R$ imply that there exists an $s\in\mathcal{N}(\mathscr{B}_R)$, unique $u\in\mathscr{U}_T(\mathscr{B})$ and unique $v\in\mathscr{V}(\mathscr{B})$ such that $W (\rho_Z^s)=R$ where $Z=\bigcup_{B\in \mathscr{B}_R}\Psi_{u(B),v(B)}[\overline{B}]$. Therefore, for every generalized reflector $R\in \mathcal{R}_1^U(T)$, we may define a unique $\mathscr{B}_R\in \mathcal{B}(U)$ such that for each $B\in \mathscr{B}_R$ there are unique $x_B\in T$ and $d_B\in(0,\infty)$ such that, for some $s\in\mathcal{N}(\mathscr{B}_R)$, $R=W\left(\rho_{Z}^s\right)$ where $Z=\bigcup_{B\in \mathscr{B}_R}\Psi_{x_B,d_{B}}[\overline{B}]$. 

Therefore, given a generalized reflector $R\in \mathcal{R}_1^U(T)$, we obtain a corresponding $\mathscr{B}_R$; for each $B\in\mathscr{B}_R$ we define unique $x_B$ and $d_B$. We also obtain an $s_R\in \mathcal{N}(\mathscr{B}_R)$ and a unique $Z_R=\bigcup_{B\in \mathscr{B}_R}\Psi_{x_B,d_{B}}[\overline{B}]$ such that $R=W(\rho_{Z_R}^{s_R})$.

Given a generalized reflector $R\in \mathcal{R}_1^U(T)$, for all $m \in \Proj[U]$ we define
\begin{equation}
    M(m)=x_B\in T
\end{equation}
where $m\rho_{Z_R}^{s_R}(m)=\Psi_{x_B,d_{B}}(m).$ Let $y_R^2(m)$ be the points of intersection between $R\setminus\{m\rho_{Z_R}^{s_R}(m)\}$ and the line segment connecting $m\rho_{Z_R}^{s_R}(m)$ to $M(m)$. 

 Given a generalized reflector $R\in \mathcal{R}_1^U(T)$, the map $\alpha_1:\Proj[U] \to T\cup R,$
\begin{equation}
    \alpha_1(m)=
\begin{cases}

M(m) &\textnormal{ if } y_R^2(m)= \varnothing\\
y_R^2(m) & \textnormal{ if } y_R^2(m)\neq \varnothing
\end{cases}
\end{equation}
is called the \textit{generalized reflector map}.  Physically speaking, a ray of light of direction $m$ originating from $\origin$ can only reach the target set if $y_R^2(m)$ is empty.

Assume we are given a nonnegative $g\in L^1(\mathbb{S}^2)$. Let us define for all Borel $X\subseteq \sphere^2$
\begin{equation}\label{s-meas2}
    \mu_{g}(X)=\int_X g(m) d\sigma(m)
\end{equation}
where $\sigma$ denotes the standard measure on $\mathbb{S}^2.$ Assume that $g\equiv 0$ outside of $\Proj[U]$. Physically speaking, $g$ is the radiance distribution of the source at $\origin$. 

In order to formulate and solve the generalized reflector problem (in the framework of weak solutions to be defined below), we need to define a measure representing the energy generated by $g$ and redistributed by a generalized reflector $R \in  \mathcal{R}_1^U(T ).$

Given a generalized reflector $R\in \mathcal{R}_1^U(T)$ and a set $\omega \subseteq T$ we define the {\it visibility set of $\omega$} as 
\begin{equation}
    V_1^U(\omega)= \bigcup_{A\in\mathscr{A}}\overline{A}\setminus \{m\in\Proj[U]|\alpha_1(m)=y_R^2(m)\}
\end{equation}
where $\mathscr{A}=\{B\in \mathscr{B}_R|x_B\in \omega\}$. We now need to show that $V_1^U(\omega)$ is measurable. Note the following definition. 

\begin{Def}\label{cone-def}
For an element $x\in\R^3$ and a set $A\subset\R^3$, let the set $C_{x,A}=\{at+x(1-t)|t\in[0,1],a\in A\}$ be the union of all line segments from $x$ to $A$ and $C_{x,A,\infty}=\{at+x(1-t)|t\in[0,\infty),a\in A\}$ be the union of all rays from $x$ that intersect $A$.
\end{Def}

We proceed with the following lemmas.

\begin{lemma}\label{e-lemma}
Let $w:\sphere^2\to (0,\infty)$ be continuous and $W(m)=mw(m)$ for all $m\in \sphere^2$. If $B$ is a Borel set of $\sphere^2$, then  $C_{\origin,W[B]}$ and $C_{\origin,W[B],\infty}$ are Borel sets of $\R^3.$ 
\end{lemma}
\begin{proof}
Recall that all Borel sets can be formed from open sets through the operations of countable union, countable intersection, and relative complement. Let $\{E_i\}$ be a countable collection of open sets of $\sphere^2$ such that through said operations, we obtain $B.$ Then given the countable collection of open sets of $\R^3$, $\{\Int(C_{\origin,W[E_i]})\}$, through the same sequence of operations we used to obtain $B$ from $\{E_i\}$, we obtain $\Int(C_{\origin,W[B]})$. Thus $C_{\origin,W[B]}$ is Borel. Therefore, assuming $W_i\equiv (i m)w(m)$ for all $m\in \sphere^2$ and $i\in(0,\infty)$, $\bigcup_{n=1}^\infty C_{\origin,W_n[B]}=C_{\origin, W[B],\infty}$ is Borel.
\end{proof}

\begin{lemma}\label{e-lemma1}
If $B$ is a Borel set of $\sphere^2$, $x\in \R^3\setminus\{\origin\}$ and $d\in(0,\infty)$, then  $C_{x,\Psi_{x,d}[B]}$ and $C_{x,\Psi_{x,d}[B],\infty}$ are Borel sets of $\R^3.$ 
\end{lemma}
\begin{proof}
Let $\sphere_x^2=\{m+x|m\in \sphere^2\}$ be the set of all unit vectors originating from $x$, i.e the unit sphere centered at $x.$  Since $x$ is another focus of the ellipsoid, there exists a continuous function $w:\sphere^2\to (0,\infty)$ such that $E_d(x)=\{mw(m)+x|m\in \sphere^2\}$. Let $W_x(m)=mw(m)+x$ and let $W(m)=mw(m)$. Note that since $B$ is Borel in $\sphere^2$, $\Psi_{x,d}[B]$ is Borel in $E_d(x)$. Thus $W_x^{-1}[\Psi_{x,d}[B]]$ is Borel in $\sphere^2$. By Lemma \ref{e-lemma}, $C_{\origin,W[W_x^{-1}[\Psi_{x,d}[B]]]}$ and $C_{\origin,W[W_x^{-1}[\Psi_{x,d}[B]]],\infty}$ are Borel sets of $\R^3.$ Thus, by translation, $C_{x,\Psi_{x,d}[B]}$ and $C_{x,\Psi_{x,d}[B],\infty}$ are Borel sets of $\R^3.$ 
\end{proof}

We can now prove the following proposition. 

\begin{prop}
Let $R$ be a generalized reflector in $\mathcal{R}_1^U(T )$. For any  set $\omega \subseteq T$ the visibility set $V_1^U(\omega)$ is Borel.
\end{prop}

\begin{proof}
We make use of the fact that sets formed from Borel sets through the operations of countable union, countable intersection, and relative complement are Borel. Recall that we obtain a $s_R\in \mathcal{N}(\mathscr{B}_R)$. Note that by the definition of a generalized reflector in $\mathcal{R}_1^U(T )$, $R=\bigcup_{B\in \mathscr{B}_R}\Psi_{x_B,d_{B}}\left[B'\right]$ where $B'=\{m\in B|m\rho_{Z_R}^{s_R}=\Psi_{x_B,d_B}(m)\}=\overline{B}\setminus\bigcup_{K\in \mathscr{K}}\overline{K}$ where $\mathscr{K}=\{A\in \mathscr{B}_R|s_R(A)<s_R(B)\}$; note that $B'$ is clearly Borel and $B\subseteq B'\subseteq \overline{B}$.

For $B\in \mathscr{B}_R$, we have that $C_{x_B,\Psi_{x_B,d_B}[B']}$ and $C_{\origin,\Psi_{x_B,d_B}[B']}$ are Borel sets by Lemmas \ref{e-lemma1} and \ref{e-lemma} respectively. Since $\mathscr{B}_R$ is countable and functions of the form $\Psi_{x_B,d_B}$ are continuous and bijective, $R$ is Borel. Thus for all $B\in \mathscr{B}_R$, the set $Q_B=C_{x_B,\Psi_{x_B,d_B}[B']}\cap (R\setminus \Psi_{x_B,d_B}[B']) $ is Borel and therefore the set $L_B=C_{x_B,Q_B,\infty}\cap \Psi_{x_B,d_B}[B']$ is Borel. Thus $\Proj[L_B]$ is Borel, as it is the preimage of $L_B$ under $\Psi_{x_B,d_B}$. Since 
\begin{equation}
    \{m\in\Proj[U]|\alpha_1(m)=y_R^2(m)\}=\bigcup_{B\in \mathscr{B}_R} \Proj[L_B],
\end{equation}
we have that $\{m\in\Proj[U]|\alpha_1(m)=y_R^2(m)\}$ is Borel and thus $V_1^U(\omega)$ is Borel.
\end{proof}

Define for any generalized reflector $R\in \mathcal{R}_1^U(T)$, 
\begin{equation}
    G_1(\omega)=\mu_g(V_1^U(\omega))
\end{equation}
which we will deem {\it the energy function of the generalized reflector problem}.

Let $F$ be a nonnegative, finite measure on the finite set $T$. We say that a generalized reflector $R\in \mathcal{R}_1^U(T)$ is a {\it weak solution to the generalized reflector problem} if the generalized reflector map $\alpha_1$ determined by $R$ is such that
\begin{equation}\label{weak1e}
    F(\omega)=G_1(\omega)\textnormal{ for any Borel set } \omega \subseteq T.
\end{equation}

It would be useful to point out the similarity of condition (\ref{weak1e}) and condition (\ref{nrg-conserv1}).

\subsection{Geometric Lemmas}
One thing that should be noted is that the definition of the generalized reflector map takes into account that there could potentially be a part of the generalized reflector that intercepts an already reflected ray before it can reach the target set. That fact inspires some key geometric lemmas.

\begin{lemma}\label{geo-lemma}
    Let $R\in \mathcal{R}_1^{U}(T)$ for some finite set $T\subset \R^{3-}$ and open set $U\subseteq R^{3+}$. For all $B\in\mathscr{B}_R$, if $m\in B$, then $\alpha_1(m)=x_B$ if and only if $y_R^2(m)=\varnothing.$
\end{lemma}
\begin{proof}
    This follows directly from the definition of the reflector map of the generalized reflector problem. 
\end{proof} 

Note the following definition. 

\begin{Def}
    
   When we say that {\bf $r$ is a ray in $C_{x,B,\infty}$}, then $r=\{at+x(1-t)|t\in[0,\infty)\}$ for some $a\in B,$ similarly if we say {\bf $r$ is a line segment in $C_{x,B}$}, then $r=\{at+x(1-t)|t\in[0,1]\}$ for some $a\in B.$
\end{Def}

\begin{lemma}\label{geo-lemma-1}
    Let $A,B\subset\sphere^2$ be disjoint sets. Then for any $x\in \R^3\setminus\{\origin\}$ and $a,b\in(0,\infty)$, $C_{x,\Psi_{x,a}[A]}\cap \Psi_{x,b}[B]=\varnothing$ and $C_{x,\Psi_{x,b}[B]}\cap \Psi_{x,a}[A]=\varnothing$ if and only if $C_{x,\Psi_{x,a}[A],\infty}\cap \Psi_{x,b}[B]=\varnothing$.
\end{lemma}
\begin{proof}
If $C_{x,\Psi_{x,a}[A],\infty}\cap \Psi_{x,b}[B]\neq\varnothing$, then there exists a ray $r$ in $C_{x,\Psi_{x,a}[A],\infty}$ such that $r$ intersects $\Psi_{x,b}[B]$. Thus, either there exists a line segment in $C_{x,\Psi_{x,a}[A]}$ that intersects $\Psi_{x,b}[B]$ and thus $C_{x,\Psi_{x,a}[A]}\cap \Psi_{x,b}[B]\neq\varnothing$, or there exists a line segment in $C_{x,\Psi_{x,b}[B]}$ that intersects $\Psi_{x,a}[A]$ and thus $C_{x,\Psi_{x,b}[B]}\cap \Psi_{x,a}[A]\neq\varnothing$. 

Conversely, if $C_{x,\Psi_{x,a}[A]}\cap \Psi_{x,b}[B]\neq\varnothing$, then there exists a line segment in $C_{x,\Psi_{x,a}[A]}$ that intersects $\Psi_{x,b}[B]$, said line segment coincides with a ray in $C_{x,\Psi_{x,a}[A],\infty}$; thus $C_{x,\Psi_{x,a}[A],\infty}\cap \Psi_{x,b}[B]\neq\varnothing$. If $C_{x,\Psi_{x,b}[B]}\cap \Psi_{x,a}[A]\neq\varnothing$, then there exists a line segment in $C_{x,\Psi_{x,b}[B]}$ that intersects $\Psi_{x,a}[A]$, said line segment coincides with a ray in $C_{x,\Psi_{x,a}[A],\infty}$; thus $C_{x,\Psi_{x,a}[A],\infty}\cap \Psi_{x,b}[B]\neq\varnothing$.
\end{proof}

These two lemmas give us the following result.
\begin{lemma}\label{geo-lemma-2}
    Assume that $U$ is an open set in $\R^{3+}$, and $T$ is a finite target set in $\R^{3-}$. Let $R\in \mathcal{R}_1^U(T)$ be a generalized reflector and $A,B\in \mathscr{B}_R$ such that $A\neq B$ and $x=x_A=x_B$. Then the following conditions are equivalent:
    \begin{enumerate}
        \item for all $m\in A$ and $m'\in B$, $\alpha_1(m)=\alpha_1(m')=x$,
        \item $C_{x,\Psi_{x,d_A}[A]}\cap \Psi_{x,d_B}[B]=\varnothing$ and $C_{x,\Psi_{x,d_B}[B]}\cap\Psi_{x,d_A}[A]=\varnothing$,
        \item $C_{x,\Psi_{x,d_A}[A],\infty}\cap \Psi_{x,d_B}[B]=\varnothing$. 
    \end{enumerate}
\end{lemma}

\begin{proof}
    (2) and (3) are equivalent by Lemma \ref{geo-lemma-1}. By Lemma \ref{geo-lemma}, for all $m\in A$ $\alpha_1(m)=x$ if and only if $y_R^2(m)=\varnothing.$ By definition, $y_R^2(m)=\varnothing$ if and only if the line segment between $\Psi_{x,d_A}(m)$ and $x$ does not intersect $R\setminus\{\Psi_{x,d_A}(m)\}$. Similarly, By Lemma \ref{geo-lemma}, for all $m'\in B$, $\alpha_1(m')=x$ if and only if $y_R^2(m')=\varnothing.$ By definition, $y_R^2(m')=\varnothing$ if and only if the line segment between $\Psi_{x,d_B}(m')$ and $x$ does not intersect $R\setminus\{\Psi_{x,d_B}(m')\}$. Therefore, statements (1) and (2) are equivalent.
\end{proof}

\subsection{Generalized Reflectors Constructed in an Open Conical Cylinder of Arbitrary Thickness}

Let $\sphere_+^2=\{m\in\sphere^2|\langle m,(0,0,1)\rangle>0\}$ be the open hemisphere of the $\sphere^2$ oriented towards the positive $z-$axis. Similarly, $\sphere_-^2=\{m\in\sphere^2|\langle m,(0,0,1)\rangle<0\}$ be the open hemisphere of the $\sphere^2$ oriented towards the negative $z-$axis. Given an open $U\subseteq \sphere_+^2$, and $\delta,z'>0$, we then define an {\it open conical cylinder of thickness $\delta$} as $\mathscr{C}_U^\delta(z')=C_{\origin,U,\infty}\cap \{(x,y,z)\in \R^3|z'+\delta>z>z'\}$. 

In this paper, given a finite target set $T\subset \R^{3-}$, we aim to construct a generalized reflector $R\in\mathcal{R}_1^{\mathscr{C}_U^\delta(z')}(T)$ that is a weak solution of the generalized reflector problem. This condition is very strict and the following strategies can potentially be applied to other kinds of open subsets in $\R^{3+}.$

We first consider the case where the target set is a single point. We proceed with the following lemmas.

\begin{lemma}\label{e-lemma2}
Let $U$ be an open set in $\sphere_+^2$ and $z',\delta>0$. Let $\{S_i\}_{i\in \mathbb{N}}$ be a countable collection of open subsets in $U$, $\{d_i\}_{i\in \mathbb{N}}$ is a countable collection of distinct positive numbers, and $x\in\R^{3-}$. Assume that each $\Psi_{x,d_i}[S_i]\subset \mathscr{C}_U^\delta(z')$ and denote $\Psi_i=\Psi_{x,d_i}[S_i].$ Then we have that 
\begin{equation}
    \Proj\left[\mathscr{C}_U^\delta(z')\setminus \bigcup_{i\in \mathbb{N}} (C_{\origin,\overline{\Psi_i},\infty}\cup C_{x,\overline{\Psi_i},\infty})\right]=\Proj\left[\mathscr{C}_U^\delta(z')\setminus \bigcup_{i\in \mathbb{N}} C_{\origin,\overline{\Psi_i},\infty}\right].
\end{equation}
\end{lemma}

\begin{proof}
Assume to the contrary that 
\begin{equation}
    \Proj\left[\mathscr{C}_U^\delta(z')\setminus \bigcup_{i\in \mathbb{N}} (C_{\origin,\overline{\Psi_i},\infty}\cup C_{x,\overline{\Psi_i},\infty})\right]\neq\Proj\left[\mathscr{C}_U^\delta(z')\setminus \bigcup_{i\in \mathbb{N}} C_{\origin,\overline{\Psi_i},\infty}\right].
\end{equation} 
Then there exists a ray $r$ in $C_{\origin,U\setminus \bigcup_{i\in \mathbb{N}}\overline{S_i},\infty}=C_{\origin,U,\infty} \setminus \bigcup_{i\in \mathbb{N}} C_{\origin,\overline{S_i},\infty}$ such that $r \cap \mathscr{C}_U^\delta(z') \subset \bigcup_{i\in \mathbb{N}} C_{x_i,\overline{S_i},\infty}.$ Equivalently, one can say that there must be a ray of direction $m\in U\setminus \bigcup_{i\in \mathbb{N}} \overline{S_i}$ originating from $\origin$ that we denote as $r$ such that $r\cap (\mathscr{C}_U^\delta(z')\setminus \bigcup_{i\in \mathbb{N}}(C_{\origin,\overline{\Psi_i},\infty}\cup C_{x,\overline{\Psi_i},\infty}))=\varnothing.$

Consider the plane $P(\alpha)=\{(x,y,z)\in \R^3|z=\alpha\}$. Let $m\in U\setminus \bigcup_{i\in\mathbb{N}}\overline{S_i}$. Assume that there exists a set $P(z')\cap \bigcup_{i\in\mathbb{N}} C_{x,\overline{\Psi_{i}},\infty}$ such that 
\begin{equation}
    \left[\left(P(z')\cap \bigcup_{i\in\mathbb{N}} C_{x,\overline{\Psi_i},\infty}\right)\setminus\left(P(z')\cap \bigcup_{i\in\mathbb{N}} C_{\origin,\overline{\Psi_i},\infty}\right)\right]\cap  C_{\origin,U,\infty}\neq \varnothing.
\end{equation}
Otherwise there does not exist a ray $r$ of direction $m\in U\setminus \bigcup_{i\in\mathbb{N}} \overline{S_i}$ originating from $\origin$ such that $r\cap (\mathscr{C}_U^\delta(z')\setminus \bigcup_{i\in \mathbb{N}}(C_{\origin,\overline{\Psi_i},\infty}\cup C_{x,\overline{\Psi_i},\infty}))=\varnothing;$ a contradiction. Thus we assume such a ray exists $r$ exists. Then $m$ must be in a direction such that there exists a $d_{\text{min}}>0$ where 
\begin{equation}
   \Psi_{x,d_{\text{min}}}(m)\in  \left[\left(P(z')\cap \bigcup_{i\in \mathbb{N}} C_{x,\overline{\Psi_i},\infty}\right)\setminus \left(P(z')\cap \bigcup_{i\in \mathbb{N}} C_{\origin,\overline{\Psi_i},\infty}\right)\right]\cap C_{\origin,U,\infty}.
\end{equation}
Since $\mathscr{C}_U^\delta(z')$ is bounded, there must also exist a $d_{\text{max}}>0$ such that 
\begin{equation}
   \Psi_{x,d_{\text{max}}}(m)\in P(z'+\delta)\cap C_{\origin,U,\infty}.
\end{equation}

Note that by our assumptions, for all $d\in (d_{\text{min}},d_{\text{max}})$, there exists an $\alpha\in \mathbb{N}$ such that the line segment between $\Psi_{x,d}(m)$ and $x$ is a subset of a line segment in $C_{x,\overline{\Psi_\alpha}}$. However, since all $\overline{\Psi_i}$ are closed, then for all $d\in [d_{\text{min}},d_{\text{max}}]$, there exists an $\alpha$ such that the line segment between $\Psi_{x,d}(m)$ and $x$ is a subset of a line segment in $C_{x,\overline{\Psi_\alpha}}$. 

\begin{mycase}
\case $d_{\text{max}}> d_i$ for all $i\in\mathbb{N}$.

Recall that by our assumptions, $\Psi_{x,d}(m)\in \bigcup_{i\in \mathbb{N}}C_{x,\overline{\Psi_i},\infty}$ for all $d\in [d_{\text{min}},d_{\text{max}}]$. However, since $\psi_{x,d_{max}}(m)>\psi_{x,d_i}(m)$ for all $i\in \mathbb{N}$, $\Psi_{x,d_{max}}(m)$ cannot reside on the interior of any ellipsoid $E_{d'}(x)$ where $d'\in \{d_i\}_{i\in \mathbb{N}}$, thus $\Psi_{x,d}(m)\not\in \bigcup_{i\in \mathbb{N}}C_{x,\overline{\Psi_i},\infty}$. A contradiction. 

\case There exists some $\alpha\in\mathbb{N}$ such that $d_{\text{max}}=d_\alpha$.

If there exists some $\alpha$ such that $d_{\text{max}}=d_\alpha$, then, since $t\Psi_{x,d_{\text{max}}}(m)+(1-t)x\not\in \mathscr{C}_U^\delta(z')$ for all $t>1$, $\Psi_{x,d_{\text{max}}}(m)$ resides on the ellipsoid $E_{d_\alpha}(x)$. Therefore $\Psi_{x,d_{\text{max}}}(m)\in \overline{\Psi_\alpha}\cap P(z'+\delta)$ and thus $m\in \overline{S_\alpha}$. A contradiction.

\case There exists some $\alpha\in\mathbb{N}$ such that $d_\alpha>d_{\text{max}}$.

Assume that $\{d_i\}_{i\in \mathbb{N}}$ is arranged such that $d_{i+1}\ge d_i$ If there exists some $\alpha$ such that $d_\alpha>d_{\text{max}}$, then there exists a ray originating from $x$ that intersects the point $\Psi_{x,d_{\text{max}}}(m)$ that also intersects a point $(x_\beta,y_\beta,z_\beta)\in\overline{\Psi_\beta}$ where $d_\beta\ge d_{max}$. The case where $d_\beta= d_{max}$ has already been covered. When $d_\beta>d_{max}$: since $x\in \R^{3-}$ and $\Psi_{x,d_{\text{max}}}(m)\in P(z'+\delta)$, this implies that $z_\beta>z'+\delta$. A contradiction.
\end{mycase}
\end{proof}

\begin{lemma}\label{e-lemma-func}
    Recall that $\sigma$ is the standard measure on $\sphere^2.$ Let $U$ be a Borel set in $\R^3\setminus\{\origin\}$ such that $\Int(U)\neq \varnothing.$ Let $x\in \R^3\setminus\{\origin\}.$ Consider the set $K(d)=\Proj[E_d(x)\cap U]$ and the corresponding function $D(d)=\sigma(K(d))$ for $d\in(0\infty)$. Then $D(d)$ cannot be identically zero.

    Furthermore, if $U$ is open, $K(d)$ is open in $\sphere^2$ for all $d\in (0,\infty)$.
\end{lemma}
\begin{proof}
    Clearly there exists a $d'\in (0,\infty)$ such that $K(d')\cap \Int(U)\neq \varnothing$. Then $E_{d'}(x)\cap \Int(U)$ is open in $E_{d'}(x)$ and thus $\Proj[E_{d'}(x)\cap \Int(U)]$ is open in $\sphere^2.$ Therefore, $D(d')\ge \sigma(\Proj[K(d')\cap \Int(U)])>0.$
\end{proof}

\begin{theorem}\label{e-prop}

Let $U$ be an open set in $\sphere_+^2$, $\delta,z'>0$, and $T=\{x\}\in\R^{3-}.$ Assume that we are given a nonnegative $g\in L^1(\sphere^2)$ where $g\equiv 0$ outside $U$. Then there exists a generalized reflector $R\in\mathcal{R}_1^{\mathscr{C}_U^\delta(z')}(T)$ such that $G_1(\{x\})=\mu_g(U)$.
\end{theorem}

\begin{proof}
For convenience, label $\mathscr{C}_*=\mathscr{C}_U^\delta(z').$ Recall that $\sigma$ is the standard measure on $\sphere^2.$

Consider the set $K_1(d)=\Proj[E_d(x)\cap \mathscr{C}_*]$ and its corresponding function $D_1(d)=\sigma(K_1(d))$ where $d\in(0,\infty).$ Note that since $\mathscr{C}_*$ is bounded, $D_1(d)\to0$ as $d\to \infty$ and $D_1(d)\to 0$ as $d\to 0.$ By construction, it is clear that $D_1$ is bounded by $0$ and $\sigma(U)$. Therefore, $D_1^{max}=\sup \{D_1(d)|d\in(0,\infty)\}$ exists and is finite, and by Lemma \ref{e-lemma-func}, $D_1^{max}>0$. 

Let $\epsilon_1\in [0,D_1^{max})$. We define $d_{max_1}$ to be a value such that $D_1(d_{max_1})=D_1^{max}-\epsilon_1$ where $\epsilon_1=0$ if $D_1^{max}\in \{D_1(d)|d\in(0,\infty)\}$. We now eliminate the parts of $U$ that had already been accounted for and the parts of $\mathscr{C}_*$ that can no longer be used: let $E_{1}=\overline{\text{Proj}[E_{d_{max_1}}(x)\cap \mathscr{C}_*]}$, $\Psi_1=\Psi_{x,d_{max_1}}[E_1]$,
\begin{equation}
    Q_2=\mathscr{C}_*\setminus (C_{x,\Psi_1,\infty} \cup C_{\origin,\Psi_1,\infty}),
\end{equation}
and $U_2=U\setminus E_1$. Note by Lemma \ref{e-lemma2}, $U_2=\Proj[Q_2]$. Let us define $K_2(d)=\Proj[\Psi_{x,d}[U_2]\cap Q_2]$ and $D_2(d)=\sigma(K_2(d))$. 

Note that since $Q_2$ is bounded, $D_2(d)\to0$ as $d\to \infty$ and $D_2(d)\to 0$ as $d\to 0.$ By construction, it is clear that $D_2$ is bounded by $0$ and $\sigma(U_2)$. Therefore, $D_2^{max}=\sup \{D_2(d)|d\in(0,\infty)\}$ exists and is finite, and by Lemma \ref{e-lemma-func}, $D_2^{max}>0$. Let $\epsilon_2\in [0,D_2^{max})$. We define $d_{max_2}$ to be a value such that $ D_1(d_{max_1})\ge D_2(d_{max_2})=D_2^{max}-\epsilon_2$.

Given that $U_1=U$ and $Q_1=\mathscr{C}_*$, we can now recursively define a sequence of functions and sets for $k\ge2$:
\begin{align}
    E_{k-1}&=\overline{\text{Proj}[\Psi_{x,d_{max_{k-1}}}[U_{k-1}]\cap Q_{k-1}]},\\
    \Psi_{k-1}&=\Psi_{x,d_{max_{k-1}}}[E_{k-1}],\\
    Q_{k}&=Q_{k-1}\setminus (C_{x,\Psi_{k-1},\infty} \cup C_{\origin,\Psi_{k-1},\infty}),\\
    U_k&=U\setminus\left(\bigcup_{j=1}^{k-1}{E_{j}}\right)=\Proj[Q_k],\\
    K_k(d)&=\text{Proj}[E_d(x)\cap Q_{k}],\\
    D_k(d)&=\sigma(K_k(d)).
\end{align}
Also, note that since $Q_k$ is bounded, $D_k(d)\to0$ as $d\to \infty$ and $D_k(d)\to 0$ as $d\to 0.$ By construction, it is clear that $D_k$ is bounded by $0$ and $\sigma(U_k)$. Therefore, $D_k^{max}=\sup \{D_k(d)|d\in(0,\infty)\}$ exists and is finite, and by Lemma \ref{e-lemma-func}, $D_k^{max}>0$. Let $\epsilon_k\in [0,D_k^{max})$. We define $d_{max_k}$ to be a value such that $D_{k-1}(d_{max_{k-1}})\ge D_k(d_{max_k})=D_k^{max}-\epsilon_k$.

Observe that the set $K_k(d)$ is open for all $d>0$. We can therefore construct a sequence 
\begin{equation}
    \left\{\sigma\left(\bigcup_{j=1}^k {E_{j}}\right)\right\}_{k=1}^\infty.
\end{equation}

\begin{claim} 
There exists $\{\epsilon_i\}_{i\in \mathbb{N}}$ such that
$\left\{\sigma\left(\bigcup_{j=1}^k {E_{j}}\right)\right\}_{k=1}^\infty$ converges to $\sigma(U).$
\end{claim}

\begin{proof}
By construction, the sequence increases monotonically and is bounded between $0$ and $\sigma(U)$; thus it converges. Assume to the contrary that for every possible $\{\epsilon_i\}_{i\in \mathbb{N}}$, $\left\{\sigma\left(\bigcup_{j=1}^k {E_{j}}\right)\right\}_{k=1}^\infty$ that converges to an $L\in(0,\sigma(U))$. Then $\sigma\left(U\setminus \bigcup_{j=1}^\infty {E_{j}}\right)=\sigma(U)-L>0.$

Consider the function
\begin{equation}
    D^*(d)=\sigma\left(\text{Proj}\left[E_d(x)\cap \lim_{j\to \infty} Q_j\right]\right).
\end{equation}
Observe that $\lim_{j\to \infty} Q_j = \mathscr{C}_*\setminus \bigcup_{i=1}^\infty(C_{x,\Psi_i,\infty} \cup C_{\origin,\Psi_i,\infty})$. Note that $\bigcup_{i=1}^\infty(C_{x,\Psi_i,\infty} \cup C_{\origin,\Psi_i,\infty})\subseteq \overline{\bigcup_{i=1}^\infty(C_{x,\Psi_i,\infty} \cup C_{\origin,\Psi_i,\infty})}$. Observe that for all $i\in \mathbb{N}$, $\overline{\Int(C_{x,\Psi_i,\infty} \cup C_{\origin,\Psi_i,\infty})}=C_{x,\Psi_i,\infty} \cup C_{\origin,\Psi_i,\infty};$ thus $\overline{\bigcup_{i=1}^\infty\Int(C_{x,\Psi_i,\infty} \cup C_{\origin,\Psi_i,\infty})}= \overline{\bigcup_{i=1}^\infty(C_{x,\Psi_i,\infty} \cup C_{\origin,\Psi_i,\infty})}.$ Thus $\lim_{j\to \infty} Q_j = \mathscr{C}_*\setminus \overline{\bigcup_{i=1}^\infty(C_{x,\Psi_i,\infty} \cup C_{\origin,\Psi_i,\infty})}$ is open and thus $\Int(\lim_{j\to \infty} Q_j)\neq \varnothing$.

Thus, by Lemma \ref{e-lemma-func}, there exists a $d'$ such that $D^*(d')>0$. By the definition of convergence, there exists an $M$ such that for all $m\ge M$, $D^*(d')>\sigma\left(\bigcup_{j=m}^\infty {E_{j}}\right).$  Note that $\sigma\left(\bigcup_{j=m}^\infty {E_{j}}\right)=\sum_{j=m}^\infty \sigma(E_j)=\sum_{j=m}^\infty D_i(d_{max_i})\ge D_m(d_{max_m})$. For all $k\in \mathbb{N}$, $D_k^{max}$ is a limit point of $\{D_k(d)|d\in(0,\infty)\}$, therefore as $\epsilon_k\to 0$, $D_k(d_{max_k})\to D_k^{max}$. Observe that $D_m(d')\ge D^*(d')$ because $\lim_{j\to \infty} Q_j\subseteq Q_m$ and $U\setminus \bigcup_{j=1}^\infty {E_{j}}\subseteq U_m.$ Therefore there exists a sequence $\{\epsilon_i\}_{i\in \mathbb{N}}$ such that $D_m(d_{max_m})\ge D^*(d').$ For this sequence $\sigma\left(\bigcup_{j=m}^\infty {E_{j}}\right)\ge D^*(d')$; a contradiction. Thus, there exists $\{\epsilon_i\}_{i\in \mathbb{N}}$ such that
$\left\{\sigma\left(\bigcup_{j=1}^k {E_{j}}\right)\right\}_{k=1}^\infty$ converges to $\sigma(U).$
\end{proof}

Let
\begin{equation}
    Z=\bigcup_{j=1}^\infty \Psi_j.
\end{equation}
For some $s\in \mathcal{N}(\mathscr{B}_Z)$, consider the generalized reflector $R=W(\rho_Z^s)\in\mathcal{R}_1^{\mathscr{C}_*}(T)$. By construction,
\begin{equation}
    C_{x,\Psi_{x,d_{max_j}}[\Int(E_j)],\infty}\cap \Psi_{x,d_{max_{j'}}}[\Int(E_{j'})]=\varnothing
\end{equation}
when $j\neq j'$. Observe that if $j',j\in \mathbb{N}$ where $j'>j$, then $d_{max_j}\neq d_{max_{j'}}$ because otherwise $E_{j'}\subseteq E_{j}$; thus $\mathscr{B}_R=\{\Int(E_j)\subset \sphere^2|j\in \mathbb{N}\}$. Thus, by Lemma \ref{geo-lemma-2}, for all $m\in B$ where $B\in \mathscr{B}_R$, we have $\alpha_1(m)=x$. Then, for any $s\in \mathcal{N}(\mathscr{B}_Z)$, the generalized reflector $R=W(\rho_Z^s)\in\mathcal{R}_1^{\mathscr{C}_*}(T)$ is a weak solution to the generalized reflector problem such that $G_1(\{x\})=\mu_g(U)$.

\end{proof}

We can now prove a result where our target set is made up of finitely many points. First, we prove the following lemma.

\begin{lemma}\label{geo-lemma-3}
    Assume that $U$ is an open set in $\R^{3+}$, and $T$ is a finite target set in $\R^{3-}$. Let $R\in \mathcal{R}_1^U(T)$ be a generalized reflector and $A,B\in \mathscr{B}_R$ such that $A\neq B$. Then the following conditions are equivalent:
    \begin{enumerate}
        \item for all $m\in A$ and $m'\in B$, $\alpha_1(m)=x_A$ and $\alpha_1(m')=x_B$,
        \item $C_{x_A,\Psi_{x_A,d_A}[A]}\cap \Psi_{x_B,d_B}[B]=\varnothing$ and $C_{x_B,\Psi_{x_B,d_B}[B]}\cap\Psi_{x_A,d_A}[A]=\varnothing$.
    \end{enumerate}
\end{lemma}

\begin{proof}
    For the case where $x_A=x_B$, we have Lemma \ref{geo-lemma-2}. We now consider the case where $x_B\neq x_B.$ By Lemma \ref{geo-lemma}, for all $m\in A$ $\alpha_1(m)=x_A$ if and only if $y_R^2(m)=\varnothing.$ By definition, $y_R^2(m)=\varnothing$ if and only if the line segment between $\Psi_{x_A,d_A}(m)$ and $x_A$ does not intersect $R\setminus\{\Psi_{x_A,d_A}(m)\}$. Similarly, By Lemma \ref{geo-lemma}, for all $m'\in B$ $\alpha_1(m')=x_B$ if and only if $y_R^2(m')=\varnothing.$ By definition, $y_R^2(m')=\varnothing$ if and only if the line segment between $\Psi_{x_B,d_B}(m')$ and $x$ does not intersect $R\setminus\{\Psi_{x_B,d_B}(m')\}$. Therefore, statements (1) and (2) are equivalent.
\end{proof}

\begin{theorem}\label{e-main2}
Let $U$ be an open set in $\sphere_+^2$, $\delta,z'>0$, and $\{x_1,\dots,x_k\}\in \R^{3-}$ where $k\ge 2$. Assume we are given a nonnegative $g\in L^1(\mathbb{S}^2)$ where $g\equiv 0$ outside $U$. Let $f_1,f_2,\dots,f_k$ be nonnegative real numbers such that
\begin{equation}
    \sum_{i=1}^k f_i=\mu_g(U).
\end{equation}
Assume that there exists $n\ge k$ disjoint open sets $B_i$ in $U$ where $\bigcup_{i\in[n]}\overline{B_i}=\overline{U}$. Also assume that there exists a collection of $k$ subsets of $[n],$ $\{A_i\}_{i\in[k]},$ such that: $A_t\cap A_{t'}=\varnothing$ where $t\neq t'$, $\bigcup_{i\in[k]} A_i=[n]$, and $\mu_g(\bigcup_{i\in A_t}B_i)=f_t.$ Suppose that for all $i\in [n]$ there exists $a_i,b_i>0$ where $z'\le a_i<a_i+b_i\le z'+\delta$ such that $\mathscr{C}_{B_i}^{b_i}(a_i)\cap {C}_{x_j,\mathscr{C}_{B_j}^{b_j}(a_j)}=\varnothing$ for all $j\in[n]\setminus \{i\}$.

Then there exists a generalized reflector in $R\in \mathcal{R}_1^{\mathscr{C}_U^\delta(z')}(T)$ such that $G_1(\{x_i\})=f_i$ for all $i\in [k]$.
\end{theorem}

\begin{proof}
For each $i \in [n]$, $\mathscr{C}_{B_i}^{b_i}(a_i)$ is open and a generalized reflector $R_i\in \mathcal{R}_1^{\mathscr{C}_{B_i}^{b_i}(a_i)}(\{x_i\})$ is constructed in the exact same way as Theorem \ref{e-prop}. By our assumptions, since $R_i\subseteq \overline{\mathscr{C}_{B_i}^{b_i}(a_i)}$, then 
\begin{equation}
    (R_i\cap \mathscr{C}_{B_i}^{b_i}(a_i)) \cap C_{x_j,R_j\cap \mathscr{C}_{B_j}^{b_j}(a_j)}=\varnothing
\end{equation} 
for all $j\in[n]\setminus \{i\}$.

Let $F=\bigcup_{i\in[k]} R_i$ and consider the generalized reflector $R=W(\rho_F)\in \mathcal{R}_1^{\mathscr{C}_U^\delta(z')}(T)$. By construction, for an $A,B\in \mathscr{B}_R$ such that $A\neq B$, we have that $C_{x_A,\Psi_{x_A,d_A}[A]}\cap \Psi_{x_B,d_B}[B]=\varnothing$ and $C_{x_B,\Psi_{x,d_B}[B]}\cap\Psi_{x_A,d_A}[A]=\varnothing$. Thus by Lemma \ref{geo-lemma-3}, for any $B\in \mathscr{B}_R$, for all $m\in B$ we have $\alpha_1(m)=x_B$. Then, for any $s\in \mathcal{N}(\mathscr{B}_F)$, the generalized reflector $R=W(\rho_F^s)\in\mathcal{R}_1^{\mathscr{C}_U^\delta(z')}(T)$ is a weak solution to the generalized reflector problem such that $G_1(\{x_i\})=f_i$ for all $i\in [k]$.
\end{proof}

We now will use Theorem \ref{e-main2} to construct a specific type of generalized reflector. Note the following definition. 
\begin{Def}\label{circdef-ellipse}
Let $k\ge 2$, $d>0$, $\xi\in(-1,0)$, and $t\in \R$. Recall that, given a point $(x,y,z)\in \R^3$, there exists $r \in [0, \infty)$, $\phi \in [0, \pi]$, $\theta \in [0, 2\pi)$, such that
\begin{align}
    x&=r\cos\theta \sin\phi\\
    y&=r\sin\theta\sin\phi\\
    z&=r\cos\phi.
\end{align}

Define the set of points $T_{k,d}^\xi(t)$ as
\begin{equation}
    \left\{\left(d\cos\left(\frac{2\pi j}k+t\right)\sin\left(\arccos(\xi)\right),d\sin\left(\frac{2\pi j}k+t\right)\sin\left(\arccos(\xi)\right),d\xi\right)|j\in I\right\}
\end{equation}
where $I=\{0,1,\dots,k-1\}.$

If we are additionally given an $i\in \{0,1,\dots,k-1\}$, we may define the set $P_{k,i}(t)\subset \sphere^2$, as 
\begin{multline}
    \left\{\left(\cos\left(\theta+\frac{\pi (2i-1)}k+t\right) \sin\phi,\sin\left(\theta+\frac{\pi (2i-1)}k+t\right)\sin\phi,\cos\phi\right)\right.\\
    \left|\phi \in [0, \pi],\theta \in \left[0, \frac{2\pi }k\right] \right\}.
\end{multline}

If $k=1,$ define the set of points $T_{1,d}^\xi(t)= \left\{(0,0,-d)\right\}$ and $P_{1,0}(t)=\sphere^2$.
\end{Def}

It is good to observe that $T_{k,d}^\xi(t)$ defines the points of a regular $k$-gon centered at the $z$-axis and that $P_{k,i}(t)$ defines a spherical wedge.

\begin{theorem}\label{rot-sym-gen-ref}
Let $\delta,z'>0$. Consider the open disk $U=\{m\in\sphere_+^2|\langle (0,0,1),m\rangle>c\}$ where $0<c<1$. Let $d_1,\dots, d_n$ be a collection of not necessarily distinct positive numbers. Let $k_1,\dots, k_n$ be a collection of not necessarily distinct positive integers. Let $\xi_1,\dots, \xi_n$ be a collection of not necessarily distinct numbers such that $\xi_i\in (-1,0)$. Let $t_1',\dots, t_n'$ be a collection of not necessarily distinct elements of $\R$. Let us denote $T_i=T_{k_i,d_i}^{\xi_i}(t_i')$ and let $T=\bigcup_{i=1}^n T_i.$

Assume that we a given a nonnegative $g\in L^1(\sphere_+^2)$ that is rotationally symmetric about the $z$-axis such that $g\equiv 0$ outside $U$. Let $f_1,\dots, f_n$ be a collection of positive numbers such that 
\begin{equation}
    \mu_g(U)=\sum_{i=1}^n f_i.
\end{equation}
Then there exists a generalized reflector $R\in\mathcal{R}_1^{\mathscr{C}_U^\delta(z')}(T)$ such that
\begin{equation}
    G_1(\{x\})=\sum_{\left\{j\in[n]|x\in T_j\right\}}\frac{f_j}{k_j}
\end{equation}
for all $x\in T.$
\end{theorem}
\begin{proof}

By the intermediate value theorem, there exists a collection of numbers $\zeta_1,\dots, \zeta_n\in[c,1)$ where $\zeta_n=c$ and $\mu_g\left(\{m\in\sphere_+^2|\langle(0,0,1),m\rangle>\zeta_i\}\right)=\sum_{j=1}^i f_j.$ Define 
\begin{equation}
    B_i=\{m\in\sphere_+^2|\langle(0,0,1),m\rangle>\zeta_{i}\}\setminus\{m\in\sphere_+^2|\langle(0,0,1),m\rangle\ge\zeta_{i-1}\}
\end{equation}
for all $i\in \{2,\dots, n\}$ and $B_1=\{ m\in\sphere_+^2|\langle(0,0,1),m\rangle>\zeta_1\}.$ Thus $\mu_g(B_i)=f_i.$ 

Consider the set $T_i$, let 
\begin{equation}
    T_i(j)=\left(d_i\cos\left(\frac{2\pi j}{k_i}+t_i'\right)\sin\left(\arccos(\xi_i)\right),d_i\sin\left(\frac{2\pi j}{k_i}+t_i'\right)\sin\left(\arccos(\xi_i)\right),d_i\xi_i\right)
\end{equation}
where $j\in\{0,\dots, k_i-1\}$ if $k\ge 2$ and $T_i(0)=(0,0,-d_i)$.

Let $P_i(j)=P_{k_i,j}(t_i')$ where $j\in\{0,\dots, k_i-1\}$, then by construction $\mu_g(B_i\cap P_i(j))=\frac{f_i}{k_i}$. Let $U_i(j)=\mathscr{C}_{B_i\cap \Int(P_i(j))}^{\frac{\delta}n}\left(z'+(i-1)\frac{\delta}{n}\right)$ where $j\in\{0,\dots, k_i-1\}$. Since $T_i(j)\in \R^{3-}$, any given line segment between a point in $U_{i}(j)$ and the point $T_i(j)$ will not intersect any set $U_{i'}(j')$ where $i'>i$ and $j'\in \{0,\dots, k_{i'}-1\}$. Therefore $C_{T_i(j), U_{i}(j)}\cap U_{i'}(j')=\varnothing$ where $i'>i$ and $j'\in \{0,\dots, k_{i'}-1\}$. Also, since $C_{\origin,P_i(j),\infty}$ is a convex set and $U_{i}(j)$, $\{T_i(j)\}$ are both subsets of $C_{\origin,P_i(j),\infty}$, we then have $C_{T_i(j),U_{i}(j)}\subset C_{\origin,P_i(j),\infty}.$ Therefore, $C_{T_i(j), U_{i}(j)}\cap U_{i}(j')=\varnothing$ where $j'\neq j$. Finally, if there exists a line segment between a point in $U_{i}(j)$ and $T_i(j)$ that intersects a $U_{i'}(j')$ where $i>i'$ and $j'\in \{0,\dots, k_{i'}-1\}$, then it must intersect $C_{\origin,\{m\in\sphere_+^2|\langle(0,0,1),m\rangle>\zeta_{i-1}\},\infty}$. However, $C_{T_i(j),U_{i}(j)}$ is disjoint from $C_{\origin,\{m\in\sphere_+^2|\langle(0,0,1),m\rangle>\zeta_{i-1}\},\infty}$ and thus $C_{T_i(j), U_{i}(j)}\cap U_{i'}(j')=\varnothing$ where $i>i'$ and $j'\in \{0,\dots, k_{i'}-1\}$. Therefore, $C_{T_i(j), U_{i}(j)}\cap U_{i'}(j')=\varnothing$ when $(i,j)\neq (i',j').$

Therefore, by Theorem \ref{e-main2}, there exists a generalized reflector $R\in\mathcal{R}_1^{C_U^\delta(z')}(T)$ such that
\begin{equation}
    G_1(\{x\})=\sum_{\left\{j\in[n]|x\in T_j\right\}}\frac{f_j}{k_j}
\end{equation}
for all $x\in T.$

\end{proof}

\section{Interpolated Reflectors}
The generalized reflector presented in the previous section might be impossible or, at best, very difficult to construct in the real world. Thus we introduce the following notion.

\begin{Def}\label{interpolatedef}
    Assume that we are given an aperture that is a connected open set $D\subseteq \sphere^2$ and a not necessarily continuous, almost everywhere differentiable function $\rho:D \to (0,\infty)$. Then an {\bf interpolated reflector} is the set $R=\partial(C_{\origin, S})\setminus \partial(C_{\origin, S,\infty})\subset \R^3$ where $S=\{m\rho(m)|m\in D\}$.
\end{Def}

It is interesting to note that, given an aperture that is a connected open set $D\subseteq \sphere^2$, a set is a reflector if and only if it is both a generalized reflector and an interpolated reflector. 

The type of interpolated reflector we construct below is a topological surface (see Chapter 4.36 in \cite{munkres2000topology}) and thus consists of one connected component instead of countably many. In a practical sense, when designing an interpolated reflector as opposed to a generalized reflector, new challenges are introduced. Thus, we settle for finding a necessary and sufficient condition for the existence of an interpolated reflector. 

We will consider the following formulation of the near-field reflector problem as a weak formulation of equation (4) from \cite{Oliker_1989} and its solutions, weak solutions. The following formulation only concerns the case where the target set is finite.

\subsection{Weak Solutions using Interpolated Reflectors}
Consider a connected open set $U\subseteq \R^{3+}$, a corresponding aperture $\Proj[U]$, and a finite target set $T\subseteq \R^{3-}$. Also, consider the set $\mathcal{R}_1^U(T)$ as defined by (\ref{gen-ref-set}). We then describe a set of interpolated reflectors
\begin{equation}\label{int-ref-set}
    \mathcal{R}_2^U(T)=\left\{\left.\partial(C_{\origin, S})\setminus \partial(C_{\origin, S,\infty})\right|S\in \mathcal{R}_1^U(T)\right\}.
\end{equation}
It is interesting to note that the interpolated reflectors in $\mathcal{R}_2^U(T)$ are all topological surfaces. 

Assume we are given an interpolated reflector $R\in \mathcal{R}_2^U(T)$. Let us define 
\begin{equation}
    \mathscr{B}_R=\left\{\Int(\Proj[E_d(x)\cap R])\subseteq \sphere^2\left|d\in(0,\infty),x\in T, \sigma(\Proj[E_d(x)\cap R])\neq 0\right.\right\}.
\end{equation}

The geometry of the ellipsoid and the definition of $\mathscr{B}_R$ imply that there exists an $s\in\mathcal{N}(\mathscr{B}_R)$, unique $u\in\mathscr{U}_T(\mathscr{B})$ and unique $v\in\mathscr{V}(\mathscr{B})$ such that $\partial(C_{\origin, W(\rho_Z^s)})\setminus \partial(C_{\origin, W(\rho_Z^s),\infty})=R$ where $Z=\bigcup_{B\in \mathscr{B}_R}\Psi_{u(B),v(B)}[\overline{B}]$. Therefore, for every interpolated reflector $R\in \mathcal{R}_2^U(T)$, we may define a unique $\mathscr{B}_R\in \mathcal{B}(U)$ such that for each $B\in \mathscr{B}_R$ there are unique $x_B\in T$ and $d_B\in(0,\infty)$ such that, for some $s\in\mathcal{N}(\mathscr{B}_R)$, $R=\partial(C_{\origin, W(\rho_Z^s)})\setminus \partial(C_{\origin, W(\rho_Z^s),\infty})$ where $Z=\bigcup_{B\in \mathscr{B}_R}\Psi_{x_B,d_{B}}[\overline{B}]$. 

Therefore, given a interpolated reflector $R\in \mathcal{R}_2^U(T)$, we obtain a corresponding $\mathscr{B}_R$; for each $B\in\mathscr{B}_R$ we define unique $x_B$ and $d_B$. We also obtain an $s_R\in \mathcal{N}(\mathscr{B}_R)$ and a unique $Z_R=\bigcup_{B\in \mathscr{B}_R}\Psi_{x_B,d_{B}}[\overline{B}]$ such that $R=\partial(C_{\origin, S})\setminus \partial(C_{\origin, S,\infty})$ where $S=W(\rho_{Z_R}^{s_R})$.

Given an interpolated reflector $R\in \mathcal{R}_2^U(T)$, let 
\begin{equation}
    M(m)=x_B\in T
\end{equation}
where $m\rho_{Z_R}^{s_R}(m)=\Psi_{x_B,d_{B}}(m).$ Let $y_R^2(m)$ be the points of intersection between $R\setminus\{m\rho_{Z_R}^{s_R}(m)\}$ and the line segment connecting $m\rho_{Z_R}^{s_R}(m)$ to $M(m)$. 

Given an interpolated reflector $R\in \mathcal{R}_2^U(T)$, the map $\alpha_2:\Proj[U] \to T\cup R,$
\begin{equation}
    \alpha_2(m)=
\begin{cases}

M(m) &\textnormal{ if } y_R^2(m)= \varnothing\\
y_R^2(m) & \textnormal{ if } y_R^2(m)\neq \varnothing
\end{cases}
\end{equation}
is called the \textit{interpolated reflector map}. Physically speaking, a ray of light of direction $m$ originating from $\origin$ can only reach the target set if $y_R^2(m)$ is empty. 

As before, we denote by $g\in L^1(\mathbb{S}^2)$ the energy density of the source $\origin$. Let us define for all Borel $X\subseteq\sphere^2$
\begin{equation}\label{s-meas}
    \mu_{g}(X)=\int_X g(m) d\sigma(m)
\end{equation}
where $\sigma$ denotes the standard measure on $\mathbb{S}^2.$ Assume that $g$ is a nonnegative function where $g\equiv 0$ outside of $\Proj[U]$. Physically speaking, $g$ is the radiance distribution of the source at $\origin$

In order to formulate and solve the interpolated reflector problem (in the framework of weak solutions to be defined below), we need to define a measure representing the energy generated by $g$ and redistributed by an interpolated reflector $R \in  \mathcal{R}_2^U(T ).$

Given a interpolated reflector $R\in \mathcal{R}_2^U(T)$ and a set $\omega \subseteq T$ we define the {\it visibility set of $\omega$} as 
\begin{equation}
    V_2^U(\omega)= \bigcup_{A\in\mathscr{A}}\overline{A}\setminus \{m\in\Proj[U]|\alpha_2(m)=y_R^2(m)\}
\end{equation}
where $\mathscr{A}=\{B\in \mathscr{B}_R|x_B\in \omega\}$. We now need to show that $V_2^U(\omega)$ is measurable.

\begin{prop}
Let $R$ be a interpolated reflector in $\mathcal{R}_2^U(T )$. For any set $\omega \subseteq T$, the visibility set $V_2^U(\omega)$ is Borel.
\end{prop}

\begin{proof}
We make use of the fact that sets formed from Borel sets through the operations of countable union, countable intersection, and relative complement are Borel. Recall that we obtain a $s_R\in \mathcal{N}(\mathscr{B}_R)$. Note that by the definition of an interpolated reflector in $\mathcal{R}_1^U(T )$, $R=\bigcup_{B\in \mathscr{B}_R}\Psi_{x_B,d_{B}}\left[B'\right]\cup S_R$ where $B'=\{m\in B|m\rho_{Z_R}^{s_R}=\Psi_{x_B,d_B}(m)\}=\overline{B}\setminus\bigcup_{K\in \mathscr{K}}\overline{K}$ where $\mathscr{K}=\{A\in \mathscr{B}_R|s_R(A)<s_R(B)\}$ and $S_R=R\setminus \bigcup_{B\in \mathscr{B}_R}\Psi_{x_B,d_{B}}[B']$. Note that $B'$ is clearly Borel and $B\subseteq B'\subseteq \overline{B}$.

For $B\in \mathscr{B}_R$, we have that $C_{x_B,\Psi_{x_B,d_B}[B']}$ and $C_{\origin,\Psi_{x_B,d_B}[B']}$ are Borel sets by Lemmas \ref{e-lemma1} and \ref{e-lemma} respectively. Note that $R$ is Borel as it is a boundary of an open set minus the boundary of another open set. Since $\mathscr{B}_R$ is countable and functions of the form $\Psi_{x_B,d_B}$ are continuous and bijective, $\bigcup_{B\in \mathscr{B}_R}\Psi_{x_B,d_{B}}\left[B'\right]$ is Borel. Therefore $S_R$ is also Borel. Thus for all $B\in \mathscr{B}_R$, the set $Q_B=C_{x_B,\Psi_{x_B,d_B}[B']}\cap (R\setminus \Psi_{x_B,d_B}[B'])$ is Borel and therefore the set $L_B=C_{x_B,Q_B,\infty}\cap \Psi_{x_B,d_B}[B']$ is Borel. Thus $\Proj[L_B]$ is Borel, as it is the preimage of $L_B$ under $\Psi_{x_B,d_B}$. Since 
\begin{equation}
    \{m\in\Proj[U]|\alpha_2(m)=y_R^2(m)\}=\bigcup_{B\in \mathscr{B}_R} \Proj[L_B],
\end{equation}
we have that $\{m\in\Proj[U]|\alpha_2(m)=y_R^2(m)\}$ is Borel and thus $V_2^U(\omega)$ is Borel.
\end{proof}

Define for any interpolated reflector $R\in \mathcal{R}_2^U(T)$, 
\begin{equation}
    G_2(\omega)=\mu_g(V_2^U(\omega))
\end{equation}
which we will deem {\it the energy function of interpolated reflector problem}.

Let $F$ be a nonnegative, finite measure on the finite set $T$. We say that an interpolated reflector $R\in \mathcal{R}_2^U(T)$ is a {\it weak solution to the interpolated reflector problem} if the interpolated reflector map $\alpha_2$ determined by $R$ is such that
\begin{equation}\label{weak2e}
    F(\omega)=G_2(\omega)\textnormal{ for any Borel set } \omega \subseteq T.
\end{equation}

\subsection{Main Results}

Here we prove a necessary and sufficient condition for the existence of weak solutions to the interpolated reflector problem. We proceed with the following lemma. 

\begin{lemma}\label{geo-lemma-4}
     Assume that $U$ is an open set in $\R^{3+}$, and $T$ is a finite target set in $\R^{3-}$. Assume we are given a nonnegative $g\in L^1(\mathbb{S}^2)$ such that $g\equiv 0$ outside $\Proj[U]$ and $g>0$ inside $\Proj[U]$. Let $R\in \mathcal{R}_1^U(T)$ be a generalized reflector, then we define the set $\mathscr{B}_R^x=\{B\in \mathscr{B}_R|x=x_B\}$ for $x\in T.$ Then for any $z,y\in T$ the following conditions are equivalent:
    \begin{enumerate}
        \item for all $m\in \bigcup_{A\in \mathscr{B}_R^z}A$ and $m'\in \bigcup_{B\in \mathscr{B}_R^y}B$, $\alpha_1(m)=z$ and $\alpha_1(m')=y$,
        \item $C_{z,\Psi_{z,d_A}[A]}\cap \Psi_{y,d_B}[B]=\varnothing$ and 
        $C_{y,\Psi_{y,d_B}[B]}\cap\Psi_{z,d_A}[A]=\varnothing$ for all $A\in \mathscr{B}_R^z$ and $B\in \mathscr{B}_R^y$ where $A\neq B$,
        \item $C_{x_A,\Psi_{x_A,d_A}[A]}\cap \Psi_{x_B,d_B}[B]=\varnothing$ and 
        $C_{x_B,\Psi_{x_B,d_B}[B]}\cap\Psi_{x_A,d_A}[A]=\varnothing$ for all $A,B\in \mathscr{B}_R$ where $A\neq B$,
        \item $G_1(\{x\})=\mu_g(\bigcup_{A\in \mathscr{B}_R^x}A)$ and $G_1(\{y\})=\mu_g(\bigcup_{B\in \mathscr{B}_R^y}B).$
    \end{enumerate}
\end{lemma}

\begin{proof}
    $(1)\Leftrightarrow (2)$ by Lemma \ref{geo-lemma-3}, clearly $(3)\Leftrightarrow (2)$, and $(1)$ trivially implies $(4)$.

    We only need to prove that $(4)$ implies $(2)$; we prove the contrapositive. Assume that there exists an $A\in \mathscr{B}_R^x$ and a $B'\in \mathscr{B}_R^y$ such that, without loss of generality, $C_{z,\Psi_{z,d_A}[A]}\cap\Psi_{y,d_{B'}}[B']\neq \varnothing.$ Let $Q=C_{z,\Psi_{z,d_A}[A]}\cap\Psi_{y,d_{B'}}[B']$. Note that $C_{z,\Psi_{z,d_A}[A],\infty}\cap C_{z,\Psi_{y,d_{B'}}[B'],\infty}=C_{z,Q,\infty}$. Since $C_{z,\Psi_{z,d_A}[A],\infty}\setminus \{\origin\}$ and $C_{z,\Psi_{y,d_{B'}}[B'],\infty}\setminus \{\origin\}$ are open, $C_{z,Q,\infty}\setminus\{\origin\}$ is open; thus $C_{z,Q,\infty}\cap E_{d_B}(y)$ is open in $E_{d_B}(y)$. Thus $\Proj[Q]\subseteq B$ is open and, since $g>0$ in $U$, $\mu_g(\Proj[Q])>0$, and thus $G_1(\{y\})\le \mu_g(\bigcup_{B\in \mathscr{B}_R^y}B)-\mu_g(\Proj[Q]).$

\end{proof}

Note the following definition.

\begin{Def}
Let $K$ be a subset of $\R^n$ where $n\ge2$ such that $\overline{K}$ is compact. The complement $U=\R^n\setminus \overline{K}$ is an open set. For sufficiently large $R>0$, the set $V=\{x|R<|x|\}$ is contained in $U$. Since $V$ is connected, there exists a connected component of $U$ that contains $V$. This is the unique unbounded connected component of $U$.

We define the {\bf exterior boundary} of $K$ as the boundary of the unbounded connected component of $\R^n\setminus \overline{K}$. We denote this as $\partial_E(K).$
\end{Def}

The following result gives a condition that is necessary and sufficient for the existence of weak solutions to the interpolated reflector problem. 

\begin{theorem}\label{nesssuff}
Let $U\subset \R^{3+}$ be a simply connected open set such that $\overline{U}\subset \R^{3+}$, and $T\subset\R^{3-}$ be a finite set. Assume we are given a nonegative $g\in L^1(\mathbb{S}^2)$ such that $g\equiv 0$ outside $\Proj[U]$ and $g>0$ inside $\Proj[U]$. Let $F$ be a measure over $T$ such that
\begin{equation}
    F(T)=\mu_g(\Proj[U]).
\end{equation}

Then there exists an interpolated reflector $R_2\in \mathcal{R}_2^{U}(T)$ that is a weak solution to the interpolated reflector problem as defined in (\ref{weak2e}) if and only if there exists a generalized reflector $R_1\in \mathcal{R}_1^{U}(T)$ that is a weak solution to the generalized reflector problem as defined in (\ref{weak1e}) where $R_1$ is a subset of a simply connected subset of 
\begin{equation}\label{e-finalthrm}
    \overline{U}\cap \partial_E\left(C_{\origin,R_1}\cup \bigcup_{B\in \mathscr{B}_{R_1}}C_{x_B,\Psi_{x_B,d_B}[B]}\right).
\end{equation}
\end{theorem}

\begin{proof}

It is clear that if $R_2\in \mathcal{R}_2^{U}(T)$ that is a weak solution to the interpolated reflector problem as defined in (\ref{weak2e}), then, for any $s\in \mathcal{N}(\mathscr{B}_{R_2})$, $R_1=W(\rho_F^s)$ where $F=\bigcup_{B\in \mathscr{B}_{R_2}}\Psi_{x_B,d_{B}}[\overline{B}]$ is a weak solution to the generalized reflector problem as defined in (\ref{weak1e}). Note that this implies that $\mathscr{B}_{R_1}=\mathscr{B}_{R_2}$. Since $g$ is positive in $\Proj[U]$, for our reflector $R_1$, by Lemma \ref{geo-lemma-4}, for all distinct $A,B\in \mathscr{B}_{R_1}$, $C_{x_A,\Psi_{x_A,d_A}[A]}\cap \Psi_{x_B,d_B}[B]=\varnothing$ and $C_{x_B,\Psi_{x_B,d_B}[B]}\cap\Psi_{x_A,d_A}[A]=\varnothing$. Therefore, for all $A\in \mathscr{B}_{R_1}$, $\Psi_{x_A,d_A}[\overline{A}]\cap \Int\left(C_{\origin,R_1}\cup \bigcup_{B\in \mathscr{B}_{R_1}}C_{x_B,\Psi_{x_B,d_B}[B]}\right)=\varnothing$. Also observe that for all $A\in \mathscr{B}_{R_1}$, $(C_{\origin, \Psi_{x_A,d_A}[A],\infty}\setminus C_{\origin, \Psi_{x_A,d_A}[A]})\cap  \Int\left(C_{\origin,R_1}\cup \bigcup_{B\in \mathscr{B}_{R_1}}C_{x_B,\Psi_{x_B,d_B}[B]}\right)=\varnothing.$

Therefore, $R_1$ is a subset of a simply connected subset of
\begin{equation}
\partial_E \left(C_{\origin,R_1}\cup \bigcup_{B\in \mathscr{B}_{R_1}}C_{x_B,\Psi_{x_B,d_B}[B]}\right).
\end{equation}
Since the interpolated reflector must be contained in $\overline{U}$, we obtain (\ref{e-finalthrm}).

Conversely, if there exists a reflector $R_1\in \mathcal{R}_1^{U}(T)$ that is a weak solution of the generalized reflector problem as defined in (\ref{weak1e}) where $R_1\subset Q$ such that $Q$ is a simply connected subset of (\ref{e-finalthrm}), then $R_2=
 \partial(C_{\origin, R_1})\setminus \partial(C_{\origin, R_1,\infty})$ is also a subset of $Q$.
\end{proof}

\section{Discussion}

In this note, with respect to the near-field reflector problem with spatial restrictions, we defined two different kinds of weak solutions. For the first weak solution, we proved, under certain assumptions, the existence of a generalized reflector where the target set is multiple points. A possible avenue for further research is to attempt to expand Theorem \ref{rot-sym-gen-ref} for different target sets, apertures, and spatial restrictions. Another idea might be to try to come up with designs such that the generalized reflectors have finitely many connected components instead of countably many. The author believes that the following statement is true.

\begin{conj}
Let $U$ be an open set in $\sphere_+^2$, $\delta,z'>0$, and $\{x_1,\dots,x_k\}\in \R^{3-}$ where $k\ge 2$. Assume we are given a positive $g\in L^1(\mathbb{S}^2)$ where $g\equiv 0$ outside $U$. Let $f_1,f_2,\dots,f_k$ be nonnegative real numbers such that
\begin{equation}
    \sum_{i=1}^k f_i=\mu_g(U).
\end{equation}
Then there exists a generalized reflector $R\in \mathcal{R}_1^{\mathscr{C}_U^\delta(z')}(T)$ such that $G_1(\{x_i\})=f_i$ for all $i\in [k]$.
\end{conj}

For the second weak solution, we proved a theorem that detailed a necessary and sufficient condition for the existence of an interpolated reflector. The advantage of our interpolated reflectors, as opposed to our generalized reflectors, is that our interpolated reflector design is a topological surface; thus it is easier to construct from an engineering perspective. An obvious avenue for further work would be to create some practically useful interpolated reflectors; using Theorem \ref{e-main2} might be useful in this regard. In fact, it would be very useful if the following conjecture is true. 

\begin{conj}
Let $U$ be a simply connected open set in $\sphere_+^2$, $\delta,z'>0$, and $\{x_1,\dots,x_k\}\in \R^{3-}$. Assume we are given a nonnegative $g\in L^1(\mathbb{S}^2)$ where $g\equiv 0$ outside $U$. Let $f_1,f_2,\dots,f_k$ be nonnegative real numbers such that
\begin{equation}
    \sum_{i=1}^k f_i=\mu_g(U).
\end{equation}

Then there exists an interpolated reflector in $R\in\mathcal{R}_2^{\mathscr{C}_U^\delta(z')}(T)$ such that $G_2(\{x_i\})=f_i$ for all $i\in [k]$.
\end{conj}

Another fruitful avenue of research might be to somehow expand these definitions of weak solutions to account for cases where the target set is not finite. Then, proving the existence of generalized and interpolated reflectors with continuous irradiance distributions.

As the reader might have noticed, we make no attempt to address the near-field reflector problem with spatial conditions with a reflector. Instead, we exclusively use generalized or interpolated reflectors. While it may be interesting to research reflectors in order to get a `stronger' solution, it is the author's view that, in general, it is not possible to construct a reflector under spatial conditions. For example, if the target set is a single point, the solution to the near-field reflector problem is an ellipsoid. However, if given spatial restrictions, a single ellipsoid, in general, cannot fit those restrictions; this is demonstrated in Theorem \ref{e-prop}. If no reflector exists for a single point, the prospects for more complicated target sets and irradiance distributions appear limited. 

\bibliography{nearfield}

\end{document}

%% file: figure_plane-reflector.pdf_tex
\begingroup%
  \makeatletter%
  \providecommand\color[2][]{%
    \errmessage{(Inkscape) Color is used for the text in Inkscape, but the package 'color.sty' is not loaded}%
    \renewcommand\color[2][]{}%
  }%
  \providecommand\transparent[1]{%
    \errmessage{(Inkscape) Transparency is used (non-zero) for the text in Inkscape, but the package 'transparent.sty' is not loaded}%
    \renewcommand\transparent[1]{}%
  }%
  \providecommand\rotatebox[2]{#2}%
  \newcommand*\fsize{\dimexpr\f@size pt\relax}%
  \newcommand*\lineheight[1]{\fontsize{\fsize}{#1\fsize}\selectfont}%
  \ifx\svgwidth\undefined%
    \setlength{\unitlength}{595.27559055bp}%
    \ifx\svgscale\undefined%
      \relax%
    \else%
      \setlength{\unitlength}{\unitlength * \real{\svgscale}}%
    \fi%
  \else%
    \setlength{\unitlength}{\svgwidth}%
  \fi%
  \global\let\svgwidth\undefined%
  \global\let\svgscale\undefined%
  \makeatother%
  \begin{picture}(1,1.41428571)%
    \lineheight{1}%
    \setlength\tabcolsep{0pt}%
    \put(0,0){\includegraphics[width=\unitlength,page=1]{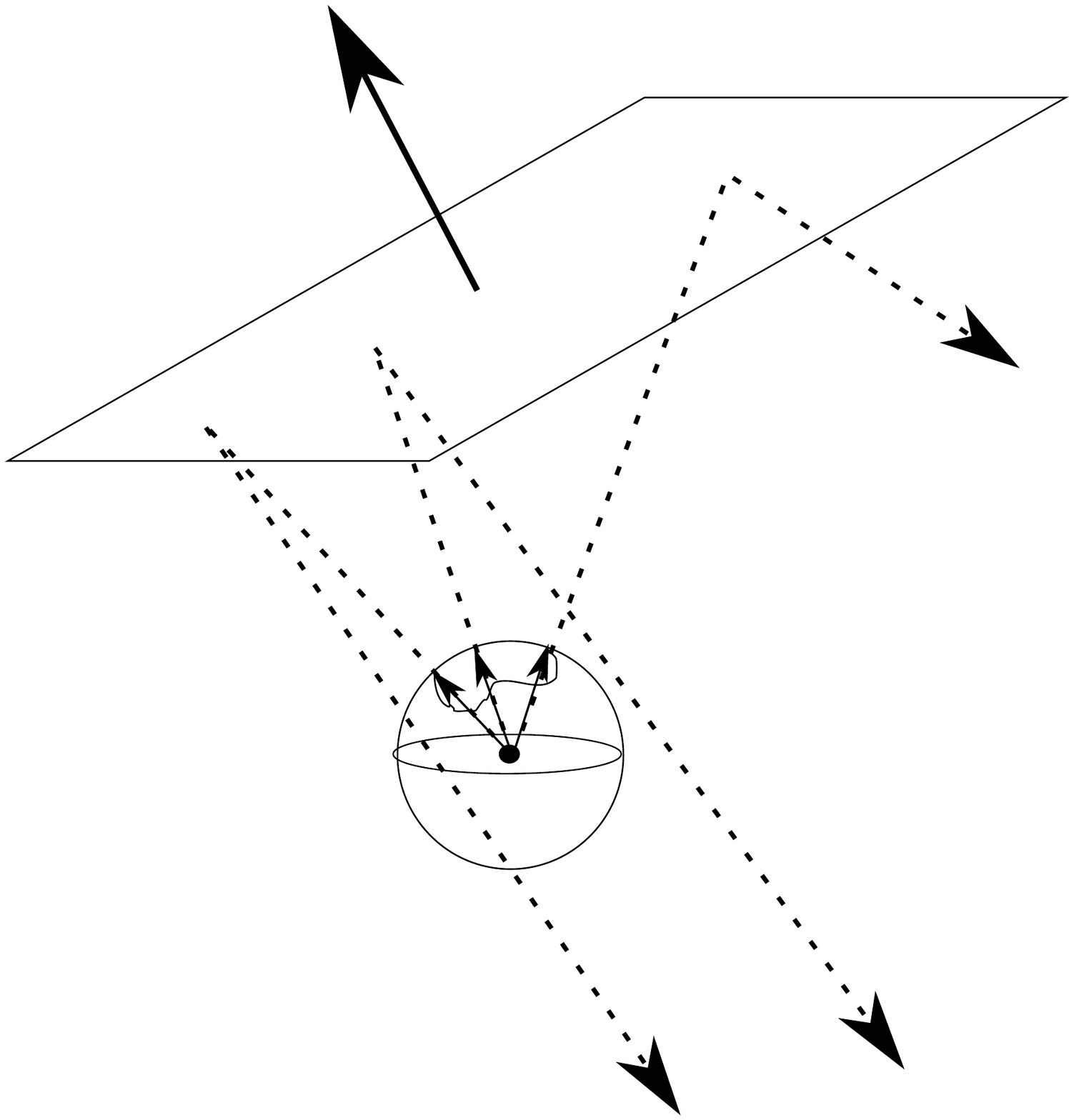}}%
    \put(0.45357143,0.37621402){\color[rgb]{0,0,0}\makebox(0,0)[lt]{\lineheight{1.25}\smash{\begin{tabular}[t]{l}$\mathcal{O}$\end{tabular}}}}%
    \put(0.71759324,0.91266716){\color[rgb]{0,0,0}\makebox(0,0)[lt]{\lineheight{1.25}\smash{\begin{tabular}[t]{l}The plane reflector $R$\end{tabular}}}}%
    \put(0.05051017,0.83428922){\color[rgb]{0,0,0}\makebox(0,0)[lt]{\lineheight{1.25}\smash{\begin{tabular}[t]{l}The surface normal to $R$\end{tabular}}}}%
    \put(0.25603447,0.20552434){\color[rgb]{0,0,0}\makebox(0,0)[lt]{\lineheight{1.25}\smash{\begin{tabular}[t]{l}Light rays going to some target set $T$\end{tabular}}}}%
  \end{picture}%
\endgroup%

%% file: figure_near-field-reflector.pdf_tex
\begingroup%
  \makeatletter%
  \providecommand\color[2][]{%
    \errmessage{(Inkscape) Color is used for the text in Inkscape, but the package 'color.sty' is not loaded}%
    \renewcommand\color[2][]{}%
  }%
  \providecommand\transparent[1]{%
    \errmessage{(Inkscape) Transparency is used (non-zero) for the text in Inkscape, but the package 'transparent.sty' is not loaded}%
    \renewcommand\transparent[1]{}%
  }%
  \providecommand\rotatebox[2]{#2}%
  \newcommand*\fsize{\dimexpr\f@size pt\relax}%
  \newcommand*\lineheight[1]{\fontsize{\fsize}{#1\fsize}\selectfont}%
  \ifx\svgwidth\undefined%
    \setlength{\unitlength}{595.27559055bp}%
    \ifx\svgscale\undefined%
      \relax%
    \else%
      \setlength{\unitlength}{\unitlength * \real{\svgscale}}%
    \fi%
  \else%
    \setlength{\unitlength}{\svgwidth}%
  \fi%
  \global\let\svgwidth\undefined%
  \global\let\svgscale\undefined%
  \makeatother%
  \begin{picture}(1,1.41428571)%
    \lineheight{1}%
    \setlength\tabcolsep{0pt}%
    \put(0,0){\includegraphics[width=\unitlength,page=1]{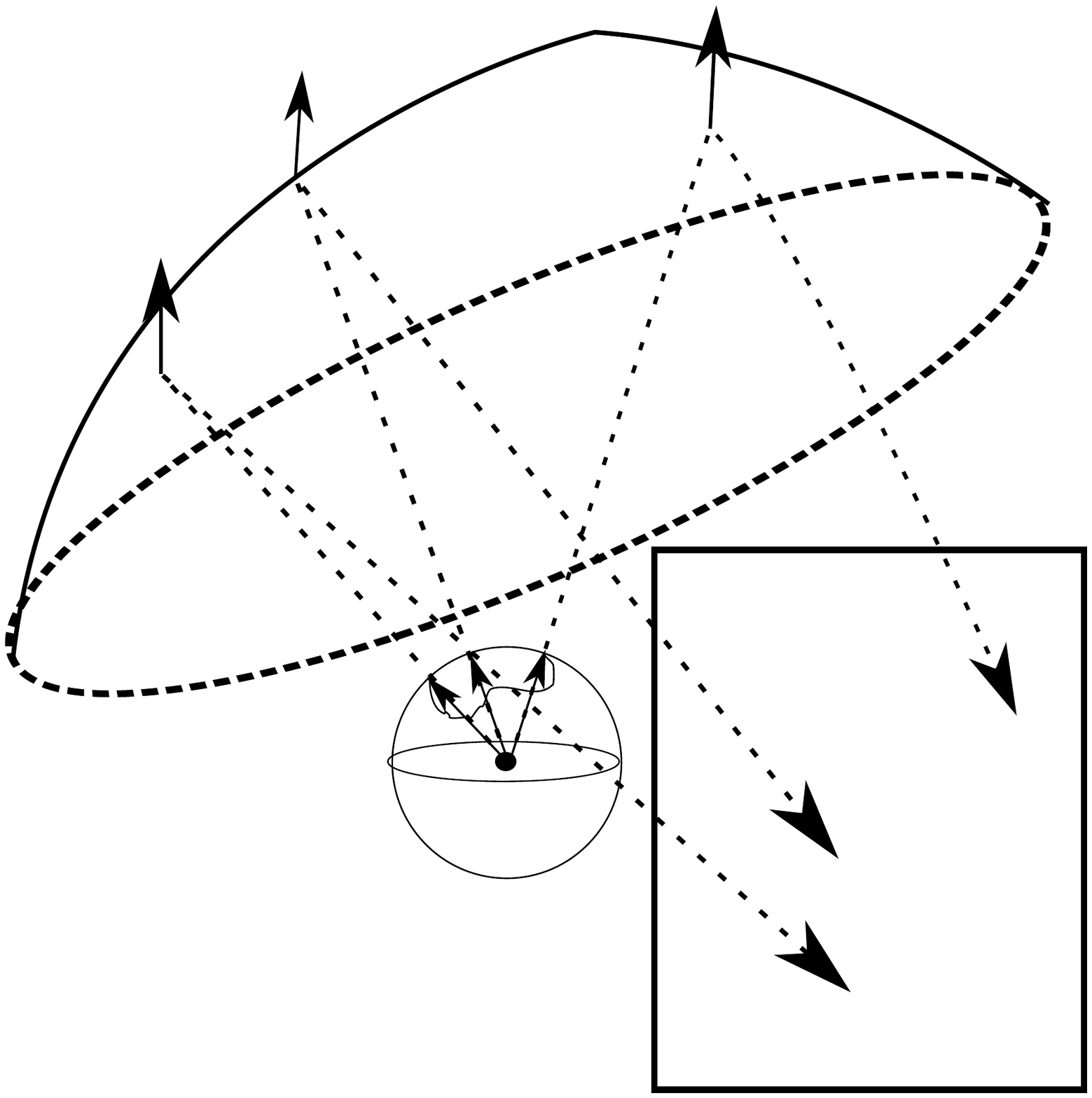}}%
    \put(0.40408162,0.37795577){\color[rgb]{0,0,0}\makebox(0,0)[lt]{\lineheight{1.25}\smash{\begin{tabular}[t]{l}$\mathcal{O}$\end{tabular}}}}%
    \put(0.07489444,0.88305773){\color[rgb]{0,0,0}\makebox(0,0)[lt]{\lineheight{1.25}\smash{\begin{tabular}[t]{l}Reflector $R$\end{tabular}}}}%
    \put(0.39537299,0.11495419){\color[rgb]{0,0,0}\makebox(0,0)[lt]{\lineheight{1.25}\smash{\begin{tabular}[t]{l}Target set $T$ with irradiance distribution $f$\end{tabular}}}}%
    \put(0.32222025,0.85867346){\color[rgb]{0,0,0}\makebox(0,0)[lt]{\lineheight{1.25}\smash{\begin{tabular}[t]{l}Surface normals to $R$\end{tabular}}}}%
    \put(0.40930684,0.63224843){\color[rgb]{0,0,0}\makebox(0,0)[lt]{\lineheight{1.25}\smash{\begin{tabular}[t]{l}Light Rays\end{tabular}}}}%
  \end{picture}%
\endgroup%